\documentclass[11pt]{amsart}
\usepackage{multirow,bigdelim}
\usepackage{amsfonts}
\usepackage{dsfont}
\usepackage{amsmath}
\usepackage{cases}
\usepackage{mathrsfs}
\usepackage{hyperref}
\usepackage{multimedia}
\usepackage{graphicx}
\usepackage{indentfirst}
\usepackage{bm}
\usepackage{amsthm}
\usepackage{mathrsfs}
\usepackage{latexsym}
\usepackage{amsmath}
\usepackage{amssymb}
\usepackage{microtype}
\usepackage{relsize}

\usepackage{hyperref}

\usepackage[numbers,sort&compress]{natbib}

\DeclareSymbolFont{SY}{U}{psy}{m}{n}
\DeclareMathSymbol{\emptyset}{\mathord}{SY}{'306}

\theoremstyle{plain}

\newtheorem{thm}{Theorem}[section]
\newtheorem{corollary}[thm]{Corollary}
\newtheorem{lemma}[thm]{Lemma}
\newtheorem{proposition}[thm]{Proposition}
\newtheorem{definition}[thm]{Definition}
\newtheorem{theorem}[thm]{Theorem}

\theoremstyle{definition}
\newtheorem{example}[thm]{Example}
\newtheorem{question}{Question}

\newtheorem{remark}[thm]{Remark}
\numberwithin{equation}{section}

\begin{document}
\title[On the metric of the jet bundle and similarity of Cowen-Douglas operators]{On the metric of the jet bundle and similarity of Cowen-Douglas operators}

\author{Kui Ji} \author{Shanshan Ji} \author{Hyun-Kyoung Kwon} \author{Jing Xu}
\curraddr[Kui Ji and Shanshan Ji]{School of Mathematics, Hebei Normal University, Shijiazhuang, Hebei 050024, China}

\curraddr[Hyun-Kyoung Kwon]{Department of Mathematics and Statistics, University at Albany-State University of New York, Albany, NY, 12222, USA}

\curraddr[Jing Xu]{School of Mathematics and Science, Hebei GEO University, Shijiazhuang 050031, China}

\email[K. Ji]{jikui@hebtu.edu.cn}
\email[S. Ji]{jishanshan15@outlook.com}
\email[H. Kwon]{hkwon6@albany.edu}
\email[J. Xu]{xujingmath@outlook.com}

\thanks{The authors except for the third author were supported by the National Natural Science Foundation of China, Grant No. 12371129 and 12471123, and the Hebei Natural Science Foundation, Grant No. A2023205045. The second and the fourth authors appreciate the hospitality of the Department of Mathematics and Statistics at the University at Albany during their visit from Aug. 2023 to Feb. 2024.}



\begin{abstract}
The study of Cowen-Douglas operators not only involves traditional operator-theoretic tools but also concepts and results from complex geometry on holomorphic vector bundles. We make use of the ratio of the metric matrices first considered by Clark and Misra and a model theorem by Agler to describe the similarity of backward shift operators on analytic function spaces whose multiplier algebras are the space of bounded analytic functions. It is well-known that, in general, it becomes much more complicated to formulate a sufficient condition for similarity than a necessary one. We also give a sufficient condition for a Cowen-Douglas operator to be similar to the backward shift operator on the Dirichlet space with weights by introducing a condition on the jet bundle of a holomorphic vector bundle. Note that the multiplier algebras of these spaces do not coincide with the space of bounded, analytic functions as in other analytic functions spaces, requiring a different approach.
The results by M\"{u}ller on operator models related to Dirichlet shifts and by Kidane and Trent on the corona problem for the multiplier algebras of weighted Dirichlet spaces are indispensable tools in attaining this similarity result.
\end{abstract}

\maketitle

\section{Introduction}
Let ${\mathcal L}({\mathcal H})$ denote the algebra of bounded linear operators on a complex separable Hilbert space $\mathcal H$.
Recall that operators $T$ and $S$ in $\mathcal{L}(\mathcal{H})$ are \emph{unitarily equivalent} ($T\sim_{u}S$) if there is a unitary operator $U$ such that $T=U^{*}SU$. An operator $T$ is \emph{similar} to $S$ ($T\sim_{s}S$) if there is a bounded invertible operator $X$ satisfying $T=X^{-1}SX$. The problem of finding unitary or similarity invariants on an infinite-dimensional space $\mathcal{H}$ is no simple task. An effective approach to study the problem is to make use of operator models given in \cite{Agler,Agler2,AEM2002,AM2003,AA1992,CF1963,VM1988,GCR1959,GKP2022,GCR1960,NF1964}, where the shift operators on certain reproducing kernel Hilbert spaces are regarded as universal operators.

Let $\mathbb{D}$ denote the open unit disk in the complex plane $\mathbb{C}$ and let $f(z)=\sum\limits_{n=0}^{\infty}\hat{f}(n)z^{n}$ be the Taylor expansion of a holomorphic function $f$ on $\mathbb{D}$. For each real number $k$, we consider a corresponding reproducing kernel Hilbert space $\mathcal{H}_{k}$
defined by
$$\mathcal{H}_{k}:=\left\{f\in \text{Hol}(\mathbb{D}):\Vert f\Vert_{\mathcal{H}_{k}}^{2}=\sum\limits_{n=0}^{\infty}(n+1)^{k}|\hat{f}(n)|^{2}<\infty\right\}.$$
Note that we obtain well-known reproducing kernel Hilbert spaces by substituting different values for $k$: $k=1$ gives the Dirichlet space $\mathcal{D}$ (or $\mathcal{D}_{1}$), for $k\in(0,1)$, we have the weighted Dirichlet spaces $\mathcal{D}_{k}$, $k=0$ yields the Hardy space $H^{2}(\mathbb{D})$, and $k=-1$ results in the Bergman space $L^{2}(\mathbb{D})$.
The backward shift operator $M_z^*$ on a reproducing kernel Hilbert space $\mathcal{H}_{k}$ belongs to the class $\mathcal{B}_{1}(\mathbb{D})$ introduced by Cowen and Douglas in \cite{CD1,CD2}. Given an open connected subset $\Omega$ of $\mathbb{C}$, the Cowen-Douglas class $\mathcal{B}_{n}(\Omega)$ of rank $n$,
is defined as
$$\begin{array}{lll}\mathcal{B}_n(\Omega)=\{T\in \mathcal{L}(\mathcal{H}):
&(1)\,\,\Omega\subset \sigma(T)=\{w\in \mathbb{C}:T-w
\mbox{ is not invertible}\},\\
&(2)\,\,\mbox{dim ker }(T-w)=n \text{ for }w\in\Omega, \\
&(3)\,\,\bigvee\limits_{w\in \Omega}\mbox{ker }(T-w)=\mathcal{H},\text{ for a closed linear span $\bigvee$,}\\
&(4)\,\,\mbox{ran }(T-w)=\mathcal{H}\text{ for }w\in\Omega\}.
\end{array}$$

Every operator in $\mathcal{B}_{n}(\Omega)$ can be realized as the adjoint of the multiplication operator on a reproducing kernel Hilbert space consisting of holomorphic functions on $\Omega^{*}=\{w\in\mathbb{C}:\overline{w}\in\Omega\}.$
Moreover, the study of a Cowen-Douglas operator $T \in \mathcal{B}_n(\Omega)$ involves the Hermitian holomorphic vector bundle $\mathcal{E}_T$ and the curvature function $\mathcal{K}_{\mathcal{E}_T}$ defined to be
$$\mathcal{E}_T=\{(w, x) \in \Omega \times \mathcal{H}: x \in \ker(T-w)\},\quad\pi (w, x)=w,$$ and
$$\mathcal{K}_{\mathcal{E}_T}(w)=-\frac{\partial}{\partial\overline{w}}\left(h_{\mathcal{E}_T}^{-1}(w)\frac{\partial
h_{\mathcal{E}_T}(w)}{\partial w}\right),$$
where $h_{\mathcal{E}_T}(w)=\left (\langle \gamma_j(w),\gamma_i(w)\rangle \right)_{i,j=1}^{n}$ denotes the metric associated with a holomorphic frame $\{\gamma_1,\ldots,\gamma_n\}$ for $\mathcal{E}_{T}$. It is shown in \cite{CD1} that operators $T$ and $S$ in $\mathcal{B}_{n}(\Omega)$ are unitarily equivalent if and only if their complex bundles $\mathcal{E}_{T}$ and $\mathcal{E}_{S}$ are equivalent ($\mathcal{E}_{S}\sim_{u}\mathcal{E}_{T}$) as Hermitian holomorphic vector bundles, that is, there exists a unitary operator $U\in\mathcal{L}(\mathcal{H})$ such that $\mathcal{E}_{S}(w)=U(\mathcal{E}_{T}(w))$ for all $w\in\Omega$. They also proved that the curvature and its covariant derivatives form a complete set of unitary invariants. However, the similarity classification problem is much more complex and has not yet been fully resolved.

For operators $T$ and $S$ in $\mathcal{B}_{1}(\mathbb{D}),$ each having $\overline{\mathbb{D}}$ as a $K$-spectral set, the Cowen-Douglas conjecture (\cite{CD1}, $\S$4.35) states that $T$ is  similar to $S$ if and only if $\lim\limits_{\vert w\vert\rightarrow 1}\mathcal{K}_{\mathcal{E}_T}(w)/\mathcal{K}_{\mathcal{E}_S}(w)=1$.
Clark and Misra gave a positive result in \cite{CM2}, confirming that the conjecture is sufficient for the class of Toeplitz operators considered in \cite{CM1978}. Subsequently, they in \cite{CM} introduced a weighted backward shift operator to show that the conjecture is not necessary, and found that the ratio of the metrics associated with Hermitian holomorphic line bundles is a better similarity invariant than the ratio of the curvatures, the main point of the note. In addition, Clark and Misra showed that the backward shift operator $T$, with $\Vert T\Vert\leq 1$, is similar to the backward shift operator $S_{\alpha}$ with weight sequence $\left\{[(n+1)/(n+2)]^{\alpha/2}\right\}_{n\geq0}$
if and only if the ratio $a_{w}:=\frac{h_{\mathcal{E}_{T}}(w)}{h_{\mathcal{E}_{S_{\alpha}}}(w)}$ of the metric is bounded and bounded away from $0$ (\cite{CM}, Theorem 1). In the third section of this paper, we will use the ratio of the metrics and the operator realization of holomorphic vector bundles to provide necessary and sufficient conditions for the similarity of the backward shift operators on the Hardy and weighted Bergman spaces.

There have been some new developments in the past four decades on the similarity problem for the weighted backward shift operators on the spaces $\mathcal{H}_k$ \cite{CS2,JJKX,JJDG,JX2022,CK,CL,LM2023,M,ZKH}. Uchiyama used a holomorphic frame of Hermitian holomorphic vector bundles to show the similarity between a contraction and the backward shift operator on the Hardy space \cite{UM1990}.
Douglas, Treil and the third author gave a necessary and sufficient condition for an $n$-hypercontraction to be similar to the backward shift operator on weighted Bergman spaces by using the trace of the curvature \cite{DKT,KT}. Sarkar and the first author described the similarity of backward shift operators on weighted Bergman spaces using Hilbert modules \cite{JS2019}.
The third author gave a necessary condition for the similarity of the backward shift on the Dirichlet space \cite{K2016}.
Note that the original Cowen-Douglas conjecture does not include the backward shift operator in the Dirichlet space and operators similar to it, see Remark 2 in \cite{CM}. In this paper, we introduce the holomorphic jet bundle in the description, which is closely related to holomorphic vector bundles, to consider the similarity of the backward shift operators on weighted Dirichlet spaces.

In brief, there have been three different approaches to dealing with the similarity problem of operators. The first is a geometric approach used in \cite{DKT,HJK,K2016,KT} which gives an explicit formulation of similarity in terms of the trace of the curvature. The second approach is algebraic and it describes the similarity of the direct sum of strongly irreducible operators \cite{CFJ,J,JGJ}.
The third approach is due to Shields in \cite{SH} in which a necessary and sufficient condition for the similarity of weighted shift operators is presented using weight sequences. 

In \cite{QL}, Lin considered the operator realization problem of when a function $f\in H^{\infty}(\mathbb{D})$ gives rise to a Cowen-Douglas operator and proved that the tensor product $\mathcal{E}_{T}\otimes\mathcal{E}_{S}$ from operators $T\in\mathcal{B}_{m}(\mathbb{D})$ and $S\in\mathcal{B}_{n}(\mathbb{D})$ is indeed a Hermitian holomorphic eigenvector bundle over $\mathbb{D}$ of some Cowen-Douglas operator in $\mathcal{B}_{nm}(\mathbb{D})$.  We are interested in using geometric entities of the Hermitian holomorphic eigenvector bundle associated with a Cowen-Douglas operator to characterize similarity. In particular, we will use the metric of the Hermitian holomorphic eigenvector bundle and the associated holomorphic
jet bundles to give a similarity criteria.

The paper is organized as follows: in Section \ref{sec1}, we introduce some necessary notations and present preliminary material.
Section \ref{sec2} deals with some extension results, Theorem \ref{t1} and Proposition \ref{3.6}, which are given in terms of the tensor product of holomorphic vector bundles and the holomorphic vector bundles induced by the defect operator corresponding to an $n$-hypercontraction. Section \ref{sec3} includes two sufficient conditions for the operator $T \in \mathcal{L}(\mathcal{H})$ to be similar to the backward shift operators $D^{*}_{\alpha}$ on weighted Dirichlet spaces $\mathcal{D}_{\alpha}$, namely Theorems \ref{thm1} and \ref{thm2}; the techniques used for these results involve the jet bundles of holomorphic complex bundles and properties of multiplier algebras of weighted Dirichlet spaces.

\section{Preliminaries}\label{sec1}
\subsection{Multipliers}
Given a reproducing kernel Hilbert space $\mathcal{H}_{K}$ of holomorphic functions on $\mathbb{D}$ with reproducing kernel $K$, \emph{the multiplier algebra} of $\mathcal{H}_{K}$ is the set of $\phi \in \text{Hol}(\mathbb{D})$ defined as
$$\text{Mult}(\mathcal{H}_{K})=\{\phi\in \text{Hol}(\mathbb{D}): \phi f\in\mathcal{H}_{K}~\text{whenever}~f\in\mathcal{H}_{K}\}.$$
It follows from the closed graph theorem that if $\phi$ is a multiplier of $\mathcal{H}_{K}$, then $M_{\phi}$, the multiplication operator by $\phi$, is a bounded linear operator on $\mathcal{H}_{K}$ and
$$\Vert\phi \Vert_{\text{Mult}(\mathcal{H}_{K})}=\Vert M_{\phi}\Vert=\sup\limits_{\Vert f\Vert_{\mathcal{H}_{K}}=1}\Vert M_{\phi}f\Vert_{\mathcal{H}_{K}}=\sup\limits_{\Vert f\Vert_{\mathcal{H}_{K}}=1}\Vert \phi f\Vert_{\mathcal{H}_{K}}.$$
Since
$\langle f, M_{\phi}^{*}K(\cdot,w)\rangle_{\mathcal{H}_{K}}=\langle M_{\phi}f, K(\cdot,w)\rangle_{\mathcal{H}_{K}}=\phi(w)f(w)=\langle f, \overline{\phi(w)}K(\cdot,w)\rangle_{\mathcal{H}_{K}}$ for any $f\in\mathcal{H}_{K},$
$$M_{\phi}^{*}K(\cdot,w)=\overline{\phi(w)}K(\cdot,w),\quad w\in\mathbb{D},$$
and $\Vert\phi\Vert_{\infty}\leq\Vert M_{\phi}\Vert=\Vert\phi \Vert_{\text{Mult}(\mathcal{H}_{K})}.$

For the Hardy space $H^{2}(\mathbb{D})$ and the Bergman space $L^{2}(\mathbb{D})$,
$$\text{Mult}(H^{2}(\mathbb{D}))=\text{Mult}(L^{2}(\mathbb{D}))=H^{\infty}(\mathbb{D}).$$
For the multiplier algebra of the weighted Dirichlet space $\mathcal{D}_{\alpha}$ for $\alpha\in(0,1]$, it is well-known that $\text{Mult}(\mathcal{D}_{\alpha})\subsetneqq H^{\infty}(\mathbb{D})$. In fact, it is even proven in \cite{EKMR,DS1980} that $\text{Mult}(\mathcal{D}_{\alpha})\subsetneqq H^{\infty}(\mathbb{D})\cap\mathcal{D}_{\alpha}$.

Given separable Hilbert spaces $E_{1}$ and $E_{2}$, the space $\mathcal{H}_{K}\otimes E_{1}$ can be seen as a space of holomorphic functions  $f:\mathbb{D}\rightarrow E_{1}$ with Taylor series
$f(z)=\sum\limits_{n=0}^{\infty}\hat{f}(n)z^{n}$, where $\{\hat{f}(n)\}_{n=0}^{\infty}\subset E_{1}$ and $\sum\limits_{n=0}^{\infty}\Vert \hat{f}(n)\Vert^{2}_{E_{1}}\Vert z^{n}\Vert^{2}_{\mathcal{H}_{K}}<\infty$. A multiplier from $\mathcal{H}_{K}\otimes E_{1}$ to $\mathcal{H}_{K}\otimes E_{2}$ is an operator-valued function
$\Phi:\Omega\rightarrow \mathcal{L}(E_{1},E_{2})$ such that for every $f\in\mathcal{H}_{K}\otimes E_{1}$, $\Phi f\in\mathcal{H}_{K}\otimes E_{2}$. The multiplication operator $M_{\Phi}$ is also  one such that given $f\in\mathcal{H}_{K}\otimes E_{1}$, $M_{\Phi} f\in\mathcal{H}_{K}\otimes E_{2}$. Then for every $f\in\mathcal{H}_{K}\otimes E_{1}$ and $g\in E_{2}$,
 \begin{eqnarray*}
\langle f, M_{\Phi}^{*}K(\cdot,w)\otimes g\rangle_{\mathcal{H}_{K}\otimes E_{1}}
&=&\langle M_{\Phi}f, K(\cdot,w)\otimes g\rangle_{\mathcal{H}_{K}\otimes E_{2}}\\
&=&\langle\Phi(w)f(w), g\rangle_{E_{2}}\\
&=&\langle f(w), \Phi(w)^{*}g\rangle_{E_{1}}\\
&=&\langle f, K(\cdot,w)\otimes\Phi(w)^{*}g\rangle_{\mathcal{H}_{K}\otimes E_{1}}.
\end{eqnarray*}

The multiplication operator has the following important property given in \cite{AM2002}:

\begin{theorem}\label{book}
Let $\Phi:\mathbb{D}\rightarrow\mathcal{L}(E_{1},E_{2})$ be an operator-valued function.
If $\Phi\in\text{Mult}(\mathcal{H}_{K}\otimes E_{1},\mathcal{H}_{K}\otimes E_{2})$, then
$$M_{\Phi}^{*}(K(\cdot,w)\otimes g)=K(\cdot,w)\otimes \Phi(w)^{*}g,\quad w\in\mathbb{D}, g\in E_{2}.$$
Conversely, if $\Phi:\mathbb{D}\rightarrow \mathcal{L}(E_{1},E_{2})$ and the mapping
$K(\cdot,w)\otimes g\mapsto K(\cdot,w)\otimes \Phi(w)^{*}g$ extends to a bounded operator $X\in\mathcal{L}(\mathcal{H}_{K}\otimes E_{2},\mathcal{H}_{K}\otimes E_{1})$,
then $\Phi\in\text{Mult}(\mathcal{H}_{K}\otimes E_{1},\mathcal{H}_{K}\otimes E_{2})$ and $X=M_{\Phi}^{*}.$
\end{theorem}
It is clear now that any multiplier $\Phi$ is bounded and holomorphic, that is, $\Phi\in H^{\infty}_{E_{1}\rightarrow E_{2}}(\mathbb{D})$, where
$H^{\infty}_{E_{1}\rightarrow E_{2}}(\mathbb{D})$ is the space of bounded analytic functions defined on
$\mathbb{D}$ whose function values are bounded linear operators from a Hilbert space $E_{1}$ to another one $E_{2}$. When $\text{Mult}(\mathcal{H}_{K})=H^{\infty}(\mathbb{D})$, then $\text{Mult}(\mathcal{H}_{K}\otimes E_{1},\mathcal{H}_{K}\otimes E_{2})=H^{\infty}_{E_1 \rightarrow E_2}(\mathbb{D})$.

Since a bounded linear operator on the Hardy space $H^2(\mathbb{D})$ commuting with the operator of multiplication by $z$ is given by the multiplication operator by a function $\varphi\in H^{\infty}(\mathbb{D})$ with norm $\Vert\varphi\Vert_{\infty},$
the corona problem is closely related to the multipliers on reproducing kernel Hilbert spaces of analytic functions (see \cite{BTV1997,XJ1998}).
The original corona problem posed by Kakutani in 1941 asks whether the unit disk $\mathbb{D}$ is dense in the maximal ideal space of $H^{\infty}(\mathbb{D})$.
The question was later answered in the affirmative by Carleson in \cite{CL1962}.
In \cite{LH1967}, H\"{o}rmander introduced a new and a much simpler method for solving the problem which was popularized by Wolff (see \cite{GJ1985,N1986}). There have been many versions of the operator corona theorem given by Fuhrmann, Rosenblum, Tolokonnikov, Uchiyama, and Vasyunin \cite{PAF1968,MR1980,VAT1981,UM1980}. In particular, we have the following result by Vasyunin \cite{VAT1981}:
\begin{theorem}\label{Fuhrmann}
For an operator-valued function $F \in H^{\infty}_{E_1 \rightarrow E_2}(\mathbb{D})$ with $\dim E_1< \infty$ and $\dim E_2=\infty$ such that
$$F^*(w)F(w) \geq \delta, \quad w \in \mathbb{D},$$
for some $\delta>0$, there exists a function $G \in H^{\infty}_{E_2 \rightarrow E_1}(\mathbb{D})$ satisfying
$$G(w)F(w) \equiv I_{E_1},\quad w \in \mathbb{D}.$$
\end{theorem}

One can consider other multiplier algebras. For a weighted Dirichlet space $\mathcal{D}_{\alpha}$ for $\alpha\in(0,1]$, $\text{Mult}(\mathcal{D}_{\alpha})$ is a Banach algebra under the norm $\Vert\phi \Vert_{\text{Mult}(\mathcal{D}_{\alpha})}=\Vert M_{\phi} \Vert$ for $f\in\text{Mult}(\mathcal{D}_{\alpha})$. In the work \cite{TTT2013}, Kidane and Trent proved the corona theorem for $\text{Mult}(\mathcal{D}_{\alpha})$ for $\alpha\in(0,1]$. Let $\{f_{i}\}_{i=1}^{\infty}\subset\text{Mult}(\mathcal{D}_{\alpha})$ and $F=(f_{1},f_{2},\ldots)$. Define a row operator $M_{F}^{R}:\bigoplus\limits_{1}^{\infty}\mathcal{D}_{\alpha}\rightarrow\mathcal{D}_{\alpha}$
by
$$M_{F}^{R}((h_{1},h_{2},\ldots)^{T})=\sum\limits_{i=1}^{\infty}f_{i}h_{i},$$
where $(h_{1},h_{2},\ldots)^{T}\in\bigoplus\limits_{1}^{\infty}\mathcal{D}_{\alpha}$. Similarly, define a column operator
$M_{F}^{C}:\mathcal{D}_{\alpha}\rightarrow\bigoplus\limits_{1}^{\infty}\mathcal{D}_{\alpha}$
by
$$M_{F}^{C}(h)=(f_{1}h,f_{2}h,\ldots)^{T},\quad h\in\mathcal{D}_{\alpha}.$$
A result that we will later need concerning these row and column operators is as follows and is due to Kidane and Trent \cite{TTT2013}:
\begin{theorem}\label{rowcolumn}
Let $M_{F}^{C}\in \mathcal{L}\left(\mathcal{D}_{\alpha},\bigoplus\limits_{1}^{\infty}\mathcal{D}_{\alpha}\right).$
Then $M_{F}^{R}\in\mathcal{L}\left(\bigoplus\limits_{1}^{\infty}\mathcal{D}_{\alpha},\mathcal{D}_{\alpha}\right)$ and $\left\Vert M_{F}^{R}\right\Vert\leq \sqrt{10}\left\Vert M_{F}^{C}\right\Vert.$
\end{theorem}
The following corona theorem for $\text{Mult}(\mathcal{D}_{\alpha})$ is also proven in \cite{TTT2013}:

\begin{theorem}\label{tavan}
Let $\{f_{i}\}_{i=1}^{\infty}\subset\text{Mult}(\mathcal{D}_{\alpha})$.
Assume that $\left\Vert M_{F}^{C}\right\Vert\leq 1$ and $0<\delta^{2}\leq\sum\limits_{i=1}^{\infty}|f_{i}(w)|^{2}$ for all $w\in\mathbb{D}$.
Then there exists a positive number $C_{\alpha}<\infty$ and $\{g_{i}\}_{i=1}^{\infty}\subset\text{Mult}(\mathcal{D}_{\alpha})$ such that
$$\sum\limits_{i=1}^{\infty}f_{i}g_{i}=1\quad\text{and}\quad \left\Vert M_{G}^{C}\right\Vert\leq\frac{C_{\alpha}}{\delta^{4}}.$$
\end{theorem}

\subsection{Model theorem}
The original model theorem states that a contraction $T \in \mathcal{L}(\mathcal{H})$ with $\lim\limits_{j\rightarrow \infty} \Vert T^j h\Vert=0$ for all
$h \in \mathcal{H}$ is unitarily equivalent to the backward shift operator restricted to some co-invariant subspace of a vector-valued Hardy space \cite{CF1963,N1986,GCR1960,NF1964}. Numerous analogues have been proposed in \cite{Agler, Agler2,AEM2002,AM2003,AA1992,VM1988,MV1993}.
In particular, we mention the models involving the backward shift on the Hardy space, weighted Bergman space, and weighted Dirichlet space given by Sz.-Nagy-Foia\c{s} \cite{NF1964}, Agler \cite{Agler, Agler2}, and M\"{u}ller \cite{VM1988}, respectively.
The results have been previously used by Douglas, Treil and the third author to describe similarity to the backward shift on these spaces in \cite{DKT}, \cite{K2016}, and \cite{KT}.

If $n$ is a positive integer, then an \emph{$n$-hypercontraction} is an operator $T$ with
$\sum\limits_{j=0}^k (-1)^j{k \choose j}(T^*)^jT^j \geq 0$ for all $1 \leq k \leq n$. In addition, we denote $D_{T}:=\left(\sum\limits_{j=0}^n (-1)^j{n \choose j}(T^*)^jT^j\right)^{1/2}$ as \emph{the defect operator} of the $n$-hypercontraction $T$.

For a positive integer $n$, we let $\mathcal{M}_n$ be the Hilbert space with reproducing kernel
$K_n(z,w)=\frac{1}{(1-\overline{w}z)^{n}}$ for $z, w \in \mathbb{D}$. As can be easily checked, different function spaces correspond to different $n$'s: the Hardy space for $n=1$ and the weighted Bergman spaces for $n \geq 2$. The vector-valued space $\mathcal{M}_{n, E}$ for a separable Hilbert space $E$ can also be defined. The forward shift $S_{n,E}$ on $\mathcal{M}_{n, E}$ is defined as $S_{n,E}f(z)=zf(z),$
and the  backward shift operator $S^*_{n,E}$ is its adjoint. We are ready to state the following theorem by Agler \cite{Agler,Agler2}.

\begin{thm}\label{model}
For $T \in \mathcal{L}(\mathcal{H})$, there exist a Hilbert space $E$ and an
$S^*_{n, E}$-invariant subspace $\mathcal{N} \subseteq
\mathcal{M}_{n, E}$ such that $T\sim_{u}S^*_{n, E}|_{\mathcal{N}}$ if and only if
$T$ is an $n$-hypercontraction and $\lim\limits_{j\rightarrow \infty} \Vert T^j x\Vert=0$ for all $x \in \mathcal{H}$.
\end{thm}

Consider the vector-valued $l^2$ space with components from a Hilbert space $E$ with
orthonormal basis $\{e_{i}\}_{i=0}^{\infty}$. The backward shift $S^{*}_{\alpha, E} \in \mathcal{L}(l^2 \otimes E)$ with a weight sequence $\{\alpha_{i}\}_{i\geq1}$ is defined by
$$S^{*}_{\alpha,E}e_{0}=0\quad \text{and} \quad S^{*}_{\alpha, E}e_{i}=\alpha_{i}e_{i-1},\quad i\geq 1.$$
Setting $b_{i}=\prod\limits_{j=1}^{i}\alpha_{j}^{-2}$ for $i\geq 1,$ we have the following theorem by M\"{u}ller \cite{VM1988}, where we assume $\alpha_1 \geq \alpha_2 \geq \cdots >0$:

\begin{theorem}\label{lem4}
Let $T \in \mathcal{L}(\mathcal{H})$ satisfy
$\sum\limits_{s=1}^{\infty}\Vert T^{s}\Vert^{2}b_{s}\leq 1$. Then there exist a Hilbert space $E$ and a subspace $\mathcal{N} \subset l^2 \otimes E$ with $S_{\alpha, E}^* \mathcal{N}\subset \mathcal{N}$ such that
$T\sim_{u}S^*_{\alpha,E}|_{\mathcal{N}}.$
\end{theorem}

For general unilateral shift operators, the work of Shields given in \cite{SH} is quite extensive. The result that we need the most in this paper is the next theorem, which can also be used to give an easy proof that the backward shift operators on the Hardy space, the weighted Bergman spaces, and the weighted Dirichlet spaces are not similar to each other.

\begin{theorem}\label{lem11}
Let $T_1$ and $T_2$ be unilateral weighted shifts with non-zero weight sequences $\{\alpha_i\}_{i \geq 1}$ and $\{\beta_i\}_{i \geq 1}$, respectively. Then $T_1\sim_{s}T_2$ if and only if there exist constants $C_1$ and $C_2$ such that for every positive integer $l$,
$$0<C_1\leq \left |\frac{\alpha_k\alpha_{k+1}\cdots\alpha_l}{\beta_k\beta_{k+1}\cdots \beta_l} \right | \leq C_2,$$
for all $1\leq k\leq l$.
\end{theorem}

\section{Similarity in the class $\mathcal{B}_{m}(\mathbb{D})$}\label{sec2}
In this section, we study Cowen-Douglas operators whose associated Hermitian holomorphic eigenvector bundles possess a tensor structure. A relationship
between the corresponding metric matrix and similarity will be given. Let $\mathcal{E}$ be a Hermitian holomorphic vector bundle over $\mathbb{D}$ of rank $n$. One can form the metric matrix $h(w)=\left(\langle \gamma_{j}(w),\gamma_{i}(w)\rangle\right)_{i,j=1}^{n}$ using a holomorphic frame  $\{\gamma_{1},\ldots,\gamma_{n}\}$ of $\mathcal{E}$. Denote by $\{\sigma_{i}\}_{i=1}^{n}$ an orthonormal basis for $\mathbb{C}^{n}$ and let
$$E=\bigvee\limits_{w\in\mathbb{D}}\{\gamma_i(w):1\leq i\leq n\}.$$
Now define an operator-valued function $F \in H^{\infty}_{\mathbb{C}^n \rightarrow E}(\mathbb{D})$ such that $$F(w)\sigma_{i}=\gamma_{i}(w), \quad 1\leq i\leq n.$$

\begin{lemma}\label{lem3}
 The following statements are equivalent:
\begin{itemize}
  \item [(1)] $F$ is left-invertible in $H^{\infty}_{\mathbb{C}^{n}\rightarrow E}(\mathbb{D})$.
  \item [(2)] There exist constants $C_1$ and $C_2$ such that
  $$0<C_{1}I_{n}\leq h(w)\leq C_{2}I_{n}, \quad w\in\mathbb{D}.$$
  \item [(3)] For every $1 \leq i \leq n$, $\Vert \gamma_{i}(w)\Vert$ is uniformly bounded on $\mathbb{D}$ and there exist constants $A_1$ and $A_2$ such that
     $$0<A_{1}\leq \det h(w)\leq A_{2}, \quad w \in \mathbb{D}.$$
\end{itemize}
\end{lemma}
\begin{proof}
In proving the implications, we use the well-known fact that for the eigenvalues $\{\lambda_i(w)\}_{i=1}^n$ of the metric matrix
$h(w)=\left(\langle \gamma_{j}(w),\gamma_{i}(w)\rangle\right)_{i,j=1}^{n}$,
$$\sum\limits_{i=1}^{n}\lambda_{i}(w)=\sum\limits_{i=1}^{n}\Vert \gamma_{i}(w)\Vert^{2}, \quad w\in\mathbb{D}.$$

$(1)\Rightarrow(2)$: Since $F$ is left-invertible in $H^{\infty}_{\mathbb{C}^{n}\rightarrow E}(\mathbb{D})$, there is a  $G\in H^{\infty}_{E\rightarrow\mathbb{C}^{n}}(\mathbb{D})$ such that
$$G(w)F(w)= I_{\mathbb{C}^{n}},\quad w\in\mathbb{D}.$$
It is then obvious that $C_1=1/\|G\|^2$. Next, since $\Vert \gamma_{i}(w)\Vert, 1\leq i\leq n,$ are uniformly bounded on $\mathbb{D}$, $C_2=n \sup\{\|\gamma_i(w)\|, 1 \leq i \leq n\}.$

$(2)\Rightarrow(1)$:
Since $0<C_{1}I_{n}\leq h(w)\leq C_{2}I_{n},$
the function $F \in H^{\infty}_{\mathbb{C}^n \rightarrow E}(\mathbb{D})$ satisfies
\begin{equation*}\label{01}
 F(w)^{*}F(w)=\left(\langle\gamma_{j}(w),\gamma_{i}(w)\rangle\right)_{i,j=1}^{n}=h(w)
\geq C_{1}I_{n},\quad w\in\mathbb{D}.
\end{equation*}
Then by Theorem \ref{Fuhrmann}, there exists a function $G\in H^{\infty}_{E\rightarrow\mathbb{C}^{n}}(\mathbb{D})$ such that
$$G(w)F(w)= I_{\mathbb{C}^{n}}, \quad w\in\mathbb{D}.$$

$(2)\Rightarrow(3)$:
Since $0<C_{1}I_{n}\leq h(w)\leq C_{2}I_{n},$
$$C_{1}\leq\lambda_{i}(w)\leq C_{2}\quad\text{and}\quad\sum\limits_{i=1}^{n}\Vert \gamma_{i}(w)\Vert^{2}=\sum\limits_{i=1}^{n}\lambda_{i}(w)\leq nC_{2},$$ so that each $\Vert \gamma_{i}(w)\Vert, 1\leq i\leq n,$ is uniformly bounded on $\mathbb{D}$ and
$$0<C_{1}^{n}\leq \det h(w)=\prod\limits_{i=1}^{n}\lambda_{i}(w)\leq (nC_{2})^{n},\quad w\in\mathbb{D}.$$

$(3)\Rightarrow(2)$: Since
$h(w)$ is positive definite and there exists some $A>0$ such that $\Vert \gamma_{i}(w)\Vert^{2}\leq A$ for every $1\leq i\leq n$ and $w \in \mathbb{D},$
$$
\lambda_{i}(w)\geq0, \quad \prod\limits_{i=1}^{n}\lambda_{i}(w)\leq(nA)^{n},\quad\text{and}\quad h(w)\leq nAI_{n}.
$$
We claim that $C:=\inf\limits_{w\in\mathbb{D}}\{\lambda_{i}(w):1\leq i\leq n\}>0$. If not, we assume, without loss of generality that $\inf\limits_{w\in\mathbb{D}}\lambda_{1}(w)=0$. Then there is a Cauchy sequence $\{w_{k}\}_{k=1}^{\infty}\subset\mathbb{D}$ such that $\lim\limits_{k\rightarrow \infty}\lambda_{1}(w_{k})=0$. This implies that  $\prod\limits_{i=2}^{n}\lambda_{i}(w_{k})\rightarrow \infty$ as $k\rightarrow \infty$, which is obviously contradictory to $\|\gamma_i(w)\|^2 \leq A$.
Thus, $0<CI_{n}\leq h(w)\leq nA I_{n}$ for every $w\in\mathbb{D}.$
\end{proof}

Regarding the Cowen-Douglas conjecture,  Clark and Misra noted in \cite{CM} that the ratio $a_{w}$ of the metrics is more suitable as a similarity invariant for Cowen-Douglas operators. For operators $T\in\mathcal{B}_{m}(\mathbb{D})$ and $S\in\mathcal{B}_{m}(\mathbb{D})$, if
$\mathcal{E}_{T}=\mathcal{E}_{S}\otimes \mathcal{E}$ for some Hermitian holomorphic vector bundle $\mathcal{E}$ over $\mathbb{D}$, then the metrics of these holomorphic vector bundles satisfy $h_{\mathcal{E}_{T}}=h_{\mathcal{E}_{S}}\otimes h_{\mathcal{E}}.$ In particular, when $\mathcal{E}_{T}$ and $\mathcal{E}_{S}$ are holomorphic line bundles, the ratio of their metrics is the metric of holomorphic line bundle $\mathcal{E}$, and thus $a_{w}:=h_{\mathcal{E}}(w)=\frac{h_{\mathcal{E}_{T}}(w)}{h_{\mathcal{E}_{S}}(w)}$.
From the model theorems given in \cite{Agler,Agler2,AEM2002,VM1988,MV1993,N1986,NF1970} and the operator theoretic realization in \cite{QL},
we know of many operators whose associated holomorphic eigenvector bundles possess a tensor structure.
We now provide details on how to construct an $n$-hypercontractive Cowen-Douglas operator with this property, using the result of Agler in \cite{Agler2}.

\begin{lemma}\label{lemma3.2}
If $T\in\mathcal{B}_{m}(\mathbb{D})$ is an $n$-hypercontraction, then there exists a Hermitian holomorphic vector
bundle $\mathcal{E}$ over $\mathbb{D}$ of rank $m$ such that $\mathcal{E}_{T}$ and $\mathcal{E}_{S^{*}_{n}}\otimes\mathcal{E}$ are equivalent as Hermitian holomorphic vector bundles, where $S^{*}_{n}$ is the backward shift operator on $\mathcal{M}_{n}$.
\end{lemma}
\begin{proof}

Since $T\in\mathcal{L}(\mathcal{H})$ is an $n$-hypercontraction and $\bigvee\limits_{w\in \mathbb{D}}\mbox{ker }(T-w)=\mathcal{H}$, we have that $\lim\limits_{j\rightarrow \infty} \Vert T^j x\Vert=0$ for all $x \in \mathcal{H}$. From Theorem \ref{model}, $T$ is unitarily equivalent to the restriction of $S^{*}_{n, E}$ to an invariant subspace $\mathcal{N}$ of $\mathcal{M}_{n, E}$ for some Hilbert space $E,$ that is, there exists a unitary operator $V: \mathcal{H}\rightarrow\mathcal{N}$ defined as
\begin{equation*}\label{equ2}
Vx=\sum\limits_{k=0}^{\infty}\frac{z^k}{\|z^k\|_{\mathcal{M}_{n}}^2}\otimes D_{T}T^{k} x,
\end{equation*}
such that $\mathcal{N}=\overline{\text{ran}V},$ $E=\overline{\text{ran}D_{T}}$ and $VT=(S^{*}_{n, E}|_{\mathcal{N}})V$, where $D_{T}$ is the defect operator of the $n$-hypercontraction $T.$
One of the main results in \cite{CD1}, Theorem 1.14, shows that two Cowen-Douglas operators are unitarily equivalent if and only if their induced holomorphic complex bundles are equivalent as Hermitian holomorphic vector bundles. Therefore, for a holomorphic frame $\gamma=\{\gamma_{1},\ldots,\gamma_{m}\}$ of $\mathcal{E}_{T}$, we have that
$$V\gamma_{i}(w)=\sum\limits_{k=0}^{\infty}\frac{z^k}{\|z^k\|_{\mathcal{M}_{n}}^2}\otimes D_{T}T^{k} \gamma_{i}(w)=
\sum\limits_{k=0}^{\infty}\frac{z^k}{\|z^k\|_{\mathcal{M}_{n}}^2}\otimes D_{T}w^{k} \gamma_{i}(w)=K_{n}(z, \overline{w})\otimes D_{T}\gamma_{i}(w),\quad 1\leq i\leq m,$$
and the eigenvector bundle $\ker(S^*_{n, E}|_{\mathcal{N}}-w)$ can be expressed as $$\ker(S^{*}_{n, E}|_{\mathcal{N}}-w)=\bigvee\{K_{n}(\cdot, \overline{w})\otimes D_{T}\gamma_{i}(w):1\leq i\leq m\}.$$

Since the eigenspaces $$\ker(T-w)=\bigvee\{\gamma_{1}(w),\ldots, \gamma_{m}(w)\}\quad\text{and} \quad\ker(S^{*}_{n}-w)=\bigvee\{K_{n}(\cdot, \overline{w})\}$$ depend analytically on $w\in\mathbb{D}$, we can say the same for the subspace $\mathcal{E}_{D_{T}}(w)=\bigvee\{D_{T}\gamma_{i}(w):1\leq i\leq m\}$. Therefore,
$\mathcal{E}_{D_{T}}=\{(w,x)\in\mathbb{D}\times\mathcal{E}(w):x\in\mathcal{E}_{D_{T}}(w)\}$ is an $m$-dimensional Hermitian holomorphic vector bundle over $\mathbb{D}$ with the natural projection $\pi, \pi(w,x)=w,$ and $\{D_{T}\gamma_{1},\ldots, D_{T}\gamma_{m}\}$ is a holomorphic frame of $\mathcal{E}_{D_{T}}$.
Hence, $\mathcal{E}_{T}$ and $\mathcal{E}_{S^{*}_{n}}\otimes\mathcal{E}_{D_{T}}$ are equivalent as Hermitian holomorphic vector bundles.
\end{proof}

For an $n$-hypercontraction $T\in\mathcal{B}_{m}(\mathbb{D})$, let $\mathcal{E}_{D_{T}}$ be the holomorphic complex bundle \emph{induced} by the defect operator $D_{T}$ of $T$. The precise definition is as follows:

\begin{definition}\label{lemma2}
If $T\in\mathcal{B}_{m}(\mathbb{D})$ is an $n$-hypercontraction and $\gamma=\{\gamma_{1},\ldots,\gamma_{m}\}$ is a holomorphic frame of $\mathcal{E}_{T}$,
we say that $\mathcal{E}_{D_{T}}=\{(w,x)\in\mathbb{D}\times\mathcal{E}_{D_{T}}(w):x\in\mathcal{E}_{D_{T}}(w)\}$ is a Hermitian holomorphic vector bundle over $\mathbb{D}$ of rank $m$ \emph{induced} by the defect operator $D_{T}$ of $T$,
where $\mathcal{E}_{D_{T}}(w)=\bigvee\{D_{T}\gamma_{i}(w):1\leq i\leq m\}$.
\end{definition}

For an $n$-hypercontraction in $\mathcal{B}_{1}(\mathbb{D})$, Theorem 1 in \cite{CM} and Theorem 3.4 in \cite{JJ2022} show that the ratio of the metrics of holomorphic line bundles is a similarity invariant for the backward shift operators on the Hardy and weighted Bergman spaces. More generally, for the similarity of operators in $\mathcal{B}_{m}(\mathbb{D})$, we use the multiplier algebra and Lemma \ref{lemma3.2} to give the following theorem, the main theorem of this section, that provides a similarity criteria for an $n$-hypercontraction $T\in\mathcal{B}_{m}(\mathbb{D})$ in terms of the metric $h_{\mathcal{E}_{D_{T}}}$ of the holomorphic complex bundle $\mathcal{E}_{D_{T}}$:

\begin{theorem}\label{t1}
For an $n$-hypercontraction $T\in\mathcal{B}_{m}(\mathbb{D})$, the following statements are equivalent:
\begin{itemize}
  \item [(1)] The operator $T$ is similar to the backward shift operator $S^{*}_{n, \mathbb{C}^{m}}$ on $\mathcal{M}_{n, \mathbb{C}^{m}}.$
  \item [(2)] There exist constants $C_1$ and $C_2$ such that the metric matrix $h_{\mathcal{E}_{D_{T}}}(w)$ of $\mathcal{E}_{D_{T}}$ in some holomorphic frame satisfies
  $$0<C_{1}I_{m}\leq h_{\mathcal{E}_{D_{T}}}(w)\leq C_{2}I_m,\quad w \in \mathbb{D}.$$
  \item [(3)] There exists a holomorphic frame $\widetilde{\gamma}=\{\widetilde{\gamma}_{1},\ldots,\widetilde{\gamma}_{m}\}$ of $\mathcal{E}_{D_{T}}$ such that each $\Vert\widetilde{\gamma}_{i}(w)\Vert$ is uniformly bounded on $\mathbb{D}$, and that the metric matrix $h_{\mathcal{E}_{D_{T}}}(w)$ under $\widetilde{\gamma}$ satisfies
      $$0<A_{1}\leq \det h_{\mathcal{E}_{D_{T}}}(w)\leq A_{2},\quad w \in \mathbb{D},$$
      for some constants $A_1$ and $A_2$.
\end{itemize}
\end{theorem}
\begin{proof}
Since $T\in\mathcal{B}_{m}(\mathbb{D})$ is an $n$-hypercontraction, by Lemma \ref{lemma3.2}, we know that $$\mathcal{E}_{D_{T}}=\{(w,x)\in\mathbb{D}\times\mathcal{E}_{D_{T}}(w):x\in\mathcal{E}_{D_{T}}(w)\}$$ is a Hermitian holomorphic vector bundle over $\mathbb{D}$ of rank $m$, $\{D_{T}\gamma_{1},\ldots, D_{T}\gamma_{m}\}$ is a holomorphic frame of $\mathcal{E}_{D_{T}}$ for every holomorphic frame $\gamma=\{\gamma_{1},\ldots,\gamma_{m}\}$ of $\mathcal{E}_{T}$,
and $\mathcal{E}_{T}\sim_{u}\mathcal{E}_{S^{*}_{n}}\otimes \mathcal{E}_{D_{T}}$.
Since the equivalence $(2)\Leftrightarrow(3)$ was already explained above in Lemma \ref{lem3}, it remains to show $(3)\Rightarrow(1)\Rightarrow(2)$ in order to complete the proof.

$(3)\Rightarrow(1)$: Using the assumption that $\Vert D_{T}\gamma_{i}(w)\Vert, 1\leq i\leq m,$ are uniformly bounded on $\mathbb{D}$,
we first define a function $F\in H^{\infty}_{\mathbb{C}^{m}\rightarrow E}(\mathbb{D})$ by $$F(w)(\sigma_i)= D_{T}\gamma_{i}(w),\quad 1\leq i\leq m,$$
where $\{\sigma_{i}\}_{i=1}^{m}$ is an orthonormal basis of $\mathbb{C}^{m}$ and $E=\bigvee\limits_{w\in\mathbb{D}}\{D_{T}\gamma_i(w):1\leq i\leq m\}$.
Then by Lemma \ref{lem3}, there exists a function $G\in H^{\infty}_{E\rightarrow\mathbb{C}^{m}}(\mathbb{D})$ such that
$$G(w)F(w)=I_{\mathbb{C}^{m}},\quad w \in \mathbb{D}.$$
Next we let $F^{\#}(w):=F(\overline{w})$ and $G^{\#}(w):=G(\overline{w})$ and note that
$$(F^{\#})^ * \in H^{\infty}_{E \rightarrow\mathbb{C}^{m}}(\mathbb{D})=\text{Mult}(\mathcal{M}_{n,E},\mathcal{M}_{n,\mathbb{C}^{m}})\quad\text{and}\quad
(G^{\#})^ * \in H^{\infty}_{\mathbb{C}^{m} \rightarrow E}(\mathbb{D})=\text{Mult}(\mathcal{M}_{n,\mathbb{C}^{m}},\mathcal{M}_{n, E} ).$$
Then from Theorem \ref{book}, we know that for $1 \leq i \leq m$,
$$M^*_{(F^{\#})^*}(K_{n}(\cdot,\overline{w}) \otimes \sigma_{i})=K_{n}(\cdot,\overline{w}) \otimes F^{\#}(\overline{w})(\sigma_{i})=K_{n}(\cdot,\overline{w}) \otimes F(w)(\sigma_{i})=K_{n}(\cdot,\overline{w}) \otimes D_{T}\gamma_{i}(w),$$
and
$$M^{*}_{(G^{\#})^{*}}M^{*}_{(F^{\#})^{*}}(K_{n}(\cdot,\overline{w}) \otimes \sigma_{i})=M^{*}_{(G^{\#})^{*}}(K_{n}(\cdot,\overline{w}) \otimes F(w)(\sigma_{i}))=K_{n}(\cdot,\overline{w}) \otimes G(w)F(w)(\sigma_{i})=K_{n}(\cdot,\overline{w}) \otimes \sigma_{i}.$$
Hence, $$M^{*}_{(G^{\#})^{*}}M^{*}_{(F^{\#})^{*}}=I_{\mathcal{M}_{n,\mathbb{C}^{m}}}.$$

Since $\mathcal{M}_{n,\mathbb{C}^{m}}=\bigvee\limits_{w\in\mathbb{D}}\{K_{n}(\cdot,\overline{w}) \otimes \sigma_{i}:1\leq i\leq m\}$ and the operator $M^{*}_{(F^{\#})^{*}}$ is left-invertible, we conclude that $\text{ran }M^{*}_{(F^{\#})^{*}}=\bigvee\limits_{w\in\mathbb{D}}\{K_{n}(\cdot,\overline{w})\otimes  D_{T}\gamma_{i}(w):1\leq i\leq m\}$.
Therefore, the operator $$X:=M^{*}_{(F^{\#})^{*}}:\mathcal{M}_{n,\mathbb{C}^{m}}\rightarrow \bigvee\limits_{w\in\mathbb{D}}\{K_{n}(\cdot,\overline{w})\otimes  D_{T}\gamma_{i}(w):1\leq i\leq m\}$$
is invertible and for every $w \in \mathbb{D}$,
$$X(\mathcal{E}_{S^{*}_{n, \mathbb{C}^{m}}}(w))=(\mathcal{E}_{S^{*}_{n}}\otimes \mathcal{E}_{D_{T}})(w).$$
Considering the composite of the unitary operator $V$ from Lemma \ref{lemma3.2} with $X$, we see that the invertible operator $X^{-1}V$ satisfies $(X^{-1}V)\mathcal{E}_{T}(w)=\mathcal{E}_{S^{*}_{n, \mathbb{C}^{m}}}(w)$ for all $w\in\mathbb{D}$. It follows that $T$ is similar to $S^{*}_{n, \mathbb{C}^{m}}$.

$(1)\Rightarrow(2)$:
Since $T\sim_{s}S^{*}_{n, \mathbb{C}^{m}}$ and $\mathcal{E}_{T}\sim_{u}\mathcal{E}_{S^{*}_{n}}\otimes \mathcal{E}_{D_{T}}$, there exists a bounded invertible operator $X$ such that
$$X(\mathcal{E}_{S^{*}_{n, \mathbb{C}^{m}}}(w))=(\mathcal{E}_{S^{*}_{n}}\otimes \mathcal{E}_{D_{T}})(w),\quad w\in \mathbb{D}.$$
Thus, for some holomorphic frame $\{\widetilde{\gamma}_{1},\ldots,\widetilde{\gamma}_{m}\}$ of $\mathcal{E}_{D_{T}}$, we have
\begin{equation}\label{21}
X(K_{n}(\cdot,\overline{w})\otimes \sigma_{i})=K_{n}(\cdot,\overline{w})\otimes \widetilde{\gamma}_{i}(w),\quad 1\leq i\leq m,
\end{equation}
where $\{\sigma_{i}\}_{i=1}^{m}$ is an orthonormal basis of $\mathbb{C}^{m}.$
Since $X$ is bounded, $\Vert \widetilde{\gamma}_{i}(w)\Vert$ is uniformly bounded on $\mathbb{D}$ for each $1\leq i\leq n$, and then
$$h_{\mathcal{E}_{D_{T}}}(w)=\left(\langle \widetilde{\gamma}_{j}(w),\widetilde{\gamma}_{i}(w)\rangle\right)_{i,j=1}^{m}\leq C_{2}I_{m},\quad w\in\mathbb{D},$$
for some $C_{2}>0$.

Now define a function $F\in H^{\infty}_{\mathbb{C}^{m}\rightarrow E}(\mathbb{D})$ as
$$F(w)\sigma_{i}=\widetilde{\gamma}_{i}(w),\quad 1\leq i \leq m,$$
where $E=\bigvee\limits_{w\in\mathbb{D}}\{\widetilde{\gamma}_{i}(w):1\leq i\leq m\}$.
Obviously, the function $F^{\#}$ defined on $\mathbb{D}$ as $F^{\#}(w):=F(\overline{w})$ is such that
$(F^{\#})^*\in \text{Mult}(\mathcal{M}_{n,E},\mathcal{M}_{n,\mathbb{C}^{m}}).$
Moreover, by Theorem \ref{book},
\begin{equation}\label{22}
M^*_{(F^{\#})^*}(K_{n}(\cdot,\overline{w})\otimes\sigma_{i})=K_{n}(\cdot,\overline{w}) \otimes F^{\#}(\overline{w})(\sigma_{i})=K_{n}(\cdot,\overline{w}) \otimes F(w)(\sigma_{i})=K_{n}(\cdot,\overline{w}) \otimes\widetilde{\gamma}_{i}(w),
\end{equation}
and $\text{ran }M^*_{(F^{\#})^*}=\bigvee\limits_{w\in\mathbb{D}}\{K_{n}(\cdot,\overline{w})\otimes \widetilde{\gamma}_{i}(w): 1\leq i\leq n\}$. From  $\mathcal{M}_{n}=\bigvee\limits_{w\in\mathbb{D}}K_{n}(\cdot,\overline{w})$, (\ref{21}) and (\ref{22}), we see that
$X=M^*_{(F^{\#})^*}$.
Since $X$ is invertible,
for every $h=\lambda K_{n}(\cdot, \overline{w})\in\mathcal{M}_{n}, \lambda\in\mathbb{C}$, $g \in \mathbb{C}^{m}$, and $w\in\mathbb{D},$ there exists a constant $\delta> 0$ such that
$$\langle X^*X (h \otimes g), h \otimes g \rangle=\Big\langle M_{(F^{\#})^*}M^*_{(F^{\#})^*}(h \otimes g), h \otimes g \Big \rangle=\Vert h\Vert^2 \langle F^*(w)F(w)g, g \rangle\geq\delta^2 \|g\|^2\|h\|^2.$$
It follows that $h_{\mathcal{E}_{D_{T}}}(w)=F^*(w)F(w)=\left(\langle \widetilde{\gamma}_{j}(w),\widetilde{\gamma}_{i}(w)\rangle\right)_{i,j=1}^{m}\geq C_{1}I_{m}$ for all $w\in\mathbb{D}$ and $C_{1}:=\delta^2.$
\end{proof}

Theorem 2.4 in \cite{DKT} and Theorem 2.3 in \cite{HJK} use the curvature of holomorphic complex bundles to characterize the similarity of $n$-hypercontractions. Using Theorem \ref{t1}, we are able to give the following quick, alternative proof:

\begin{corollary}\label{lemma3}
Let $T\in\mathcal{B}_{m}(\mathbb{D})$ be an $n$-hypercontraction. Then $T$ is similar to the backward shift operator $S^{*}_{n, \mathbb{C}^{m}}$ on $\mathcal{M}_{n, \mathbb{C}^{m}}$ if and only if there exists a bounded subharmonic function $\varphi$ over $\mathbb{D}$ such that
\begin{equation*}\label{3.3}
\text{trace}\;\mathcal{K}_{S^{*}_{n, \mathbb{C}^{m}}}(w)-\text{trace}\;\mathcal{K}_{T}(w)=\frac{\partial^{2}}{\partial w\partial\overline{w}}\varphi(w), \quad w\in\mathbb{D}.
\end{equation*}

\end{corollary}
\begin{proof} Since $T\in\mathcal{B}_{m}(\mathbb{D})$ is an $n$-hypercontraction, by Lemma \ref{lemma3.2}, there exists a Hermitian holomorphic vector bundle $\mathcal{E}_{D_{T}}$ over $\mathbb{D}$ of rank $m$ such that $\mathcal{E}_{T}\sim_{u}\mathcal{E}_{S^{*}_{n}}\otimes \mathcal{E}_{D_{T}}$.

If $T$ is similar to $S^{*}_{n, \mathbb{C}^{m}}$, by Theorem \ref{t1}, there exist constants $A_1$, $A_2$, and a holomorphic frame $\gamma=\{\gamma_{1},\ldots,\gamma_{m}\}$ of $\mathcal{E}_{D_{T}}$ such that the metric $h_{\mathcal{E}_{D_{T}}}$ under $\gamma$ satisfies
$$0<A_{1}\leq \det h_{\mathcal{E}_{D_{T}}}(w)\leq A_{2},\quad w \in \mathbb{D}.$$
Setting $\det\mathcal{E}_{D_{T}}$ as the determinant bundle of $\mathcal{E}_{D_{T}}$, the metric on the determinant bundle is $h_{\det\mathcal{E}_{D_{T}}}=\det h_{\mathcal{E}_{D_{T}}},$ and we have
$$ \text{trace}\;\mathcal{K}_{S^{*}_{n, \mathbb{C}^{m}}}(w)-\text{trace}\;\mathcal{K}_{T}(w)=-\text{trace}\;\mathcal{K}_{\mathcal{E}_{D_{T}}}(w)
=-\mathcal{K}_{\det\mathcal{E}_{D_{T}}}(w)=\frac{\partial^{2}}{\partial w\partial\overline{w}}\log h_{\det\mathcal{E}_{D_{T}}}(w)
\geq0.$$
To complete the proof, one sets $\varphi(w):=\log h_{\text{det}\mathcal{E}_{D_{T}}}(w)$ on $\mathbb{D}$.

For the other implication, note that since $\mathcal{E}_T \sim_{u} \mathcal{E}_{S^{*}_{n}} \otimes \mathcal{E}_{D_{T}}$, we have
$$\text{trace}\;\mathcal{K}_T=\text{trace}\;\mathcal{K}_{S^{*}_{n, \mathbb{C}^{m}}}+\text{trace}\;\mathcal{K}_{\mathcal{E}_{D_{T}}}.$$
Thus,
$$\frac{\partial^{2}}{\partial w\partial\overline{w}}\varphi(w)=-\text{trace}\;\mathcal{K}_{\mathcal{E}_{D_{T}}}(w)
=-\mathcal{K}_{\det\mathcal{E}_{D_{T}}}(w)=\frac{\partial^{2}}{\partial w\partial\overline{w}}\log\Vert{\gamma}(w)\Vert^{2}$$ for a bounded subharmonic function $\varphi$ on $\mathbb{D}$ and a non-vanishing holomorphic frame ${\gamma}$ of the determinant bundle $\det\mathcal{E}_{D_{T}}$.
Now,
$u(w):=\log \left(\frac{\Vert{\gamma}(w)\Vert^{2}}{e^{\varphi(w)}}\right)^{\frac{1}{2}}$ is a real-valued harmonic function, and letting $\phi:=e^{u+iv}\in \text{Hol}(\mathbb{D}),$
where $v$ is the harmonic conjugate of $u$, we see that
$$|\phi(w)|=e^{u(w)}=\left(\frac{\Vert {\gamma}(w)\Vert^{2}}{e^{{\varphi}(w)}}\right)^{\frac{1}{2}}.$$
Hence,
$e^{{\varphi}(w)}=\left\Vert\frac{1}{\phi (w)}{\gamma}(w)\right\Vert^{2}$
and $\frac{1}{\phi}{\gamma}$ is a holomorphic frame of $\det\mathcal{E}_{D_{T}}$. Since $\varphi$ is bounded on $\mathbb{D}$,  there are constants $C_{1}$ and $C_{2}$ such that
$$0<C_{1}\leq e^{\varphi(w)}=\left \Vert\frac{1}{\phi(w)}\gamma(w)\right \Vert ^2\leq C_{2},\quad w\in\mathbb{D},$$
and therefore by Theorem \ref{t1} and Theorem 3.2 in \cite{KT}, $T\sim_{s}S^{*}_{n, \mathbb{C}^{m}}$.
\end{proof}

Two Cowen-Douglas operators are unitarily equivalent if and only if their corresponding holomorphic complex bundles are equivalent as Hermitian holomorphic vector bundles.  It is known from \cite{QL} that many holomorphic complex bundles with tensor products correspond to some Cowen-Douglas operator. Therefore, we use holomorphic complex bundles to give a similarity classification of Cowen-Douglas operators and obtain the following proposition. We omit the proof of it which is similar to those of Theorem \ref{t1} and Corollary \ref{lemma3}:

\begin{proposition}\label{3.6}
Let $T\in\mathcal{B}_{m}(\mathbb{D})$ and let $M_{z}^{*}\in\mathcal{B}_{1}(\mathbb{D})$ be the backward shift operator on a reproducing kernel Hilbert space $\mathcal{H}_{K}$ with $\text{Mult}(\mathcal{H}_{K})=H^{\infty}(\mathbb{D})$. If $\mathcal{E}_{T}\sim_{u}\mathcal{E}_{M_{z}^{*}}\otimes \mathcal{E}$ for some Hermitian holomorphic vector bundle $\mathcal{E}$ over $\mathbb{D}$ of rank $m$, then the following statements are equivalent:
\begin{itemize}
  \item [(1)] The operator $T$ is similar to the backward shift operator $M^{*}_{z, \mathbb{C}^{m}}$ on $\mathcal{H}_{K, \mathbb{C}^{m}}.$
  \item [(2)] There exist constants $C_1$ and $C_2$ such that the metric $h_{\mathcal{E}}(w)$ of $\mathcal{E}$ in some holomorphic frame satisfies
  $0<C_{1}I_{m}\leq h_{\mathcal{E}}(w)\leq C_{2}I_m$ for every $w \in \mathbb{D}.$
  \item [(3)] There exists a holomorphic frame $\gamma=\{\gamma_{1},\ldots,\gamma_{m}\}$ of $\mathcal{E}$ such that each $\Vert\gamma_{i}(w)\Vert, 1\leq i\leq m$, is uniformly bounded on $\mathbb{D}$, and there are constants $A_1$ and $A_2$ such that the metric $h_{\mathcal{E}}(w)$ under the holomorphic frame $\gamma$ satisfies $$0<A_{1}\leq \det h_{\mathcal{E}}(w)\leq A_{2},\quad w \in \mathbb{D}.$$
       In other words, the metric $h_{\det\mathcal{E}}$ of the determinant bundle $\det\mathcal{E}$ satisfies $0<A_{1}\leq h_{\det\mathcal{E}}(w)\leq A_{2}$ for every $w \in \mathbb{D}.$
  \item [(4)] There exists a bounded subharmonic function $\varphi$ over $\mathbb{D}$ such that
\begin{equation*}
m\mathcal{K}_{M_{z}^{*}}(w)-\text{trace}\;\mathcal{K}_{T}(w)=\frac{\partial^{2}}{\partial w\partial\overline{w}}\varphi(w), \quad w\in\mathbb{D}.
\end{equation*}
\end{itemize}
\end{proposition}

In Proposition \ref{3.6}, if the metric of the holomorphic complex bundle $\mathcal{E}$ under a holomorphic frame is a polynomial for $|w|^{2}$, with each coefficient being a non-negative real number, then we can construct numerous examples of contractive operators similar to the backward shift operators on the Hardy or weighted Bergman spaces. These examples satisfy the conditions of Theorem 1 in \cite{CM}.

\section{Similarity on weighted Dirichlet spaces}\label{sec3}
Shields, in \cite{SH}, provided a necessary and sufficient condition for the similarity of backward shift operators on Hilbert space using weight sequences.
From \cite{JKSX2022,QL}, we know when a tensor product of Hermitian holomorphic vector bundles corresponds to a bounded linear operator. The goal of this section is to use the holomorphic jet bundle to describe the similarity of backward shifts on weighted Dirichlet spaces.

Let $\mathcal{E}$ be a Hermitian holomorphic line bundle over $\mathbb{D}$ and let $\gamma$ be a non-vanishing holomorphic cross-section of $\mathcal{E}$. For $k=0, 1, 2, \cdots$, we associate to $\mathcal{E}$ a $(k+1)$-dimensional Hermitian holomorphic vector bundle $\mathcal{J}_k(\mathcal{E})$ called the holomorphic \emph{$k$-jet bundle} of $\mathcal{E}$. The work \cite{CD1,GG1973,DKK2014} are good references for $k$-jet bundles. Suppose that for each $w \in \mathbb{D}$, $\gamma(w),\gamma^{'}(w),\gamma^{(2)}(w),\ldots,\gamma^{(k)}(w)$ are independent. The $k$-jet bundle $\mathcal{J}_{k}(\mathcal{E})$ of the bundle $\mathcal{E}$ has an associated holomorphic frame $\mathcal{J}_{k}(\gamma)=\{\gamma,\gamma^{'},\gamma^{(2)},\ldots,\gamma^{(k)}\}$.
The metric $h_{\mathcal{E}}(w)=\langle\gamma(w), \gamma(w)\rangle$ for $\mathcal{E}$ induces a metric $\mathcal{J}_{k}(h_{\mathcal{E}})(w)$ for $\mathcal{J}_{k}(\mathcal{E})$ of the form
\begin{equation*}
\mathcal{J}_{k}(h_{\mathcal{E}})(w)=
\begin{pmatrix}
h_{\mathcal{E}}(w)&\frac{\partial}{\partial w}h_{\mathcal{E}}(w)&\cdots&\frac{\partial^{k}}{\partial w^{k}}h_{\mathcal{E}}(w)\\
\frac{\partial}{\partial \overline{w}}h_{\mathcal{E}}(w)&\frac{\partial^{2}}{\partial w\partial \overline{w}}h_{\mathcal{E}}(w)&\cdots&\frac{\partial^{k+1}}{\partial w^{k}\partial \overline{w}}h_{\mathcal{E}}(w)\\
\vdots&\vdots&\ddots&\vdots\\
\frac{\partial^{k}}{\partial \overline{w}^{k}}h_{\mathcal{E}}(w)&\frac{\partial^{k+1}}{\partial w\partial \overline{w}^{k}}h_{\mathcal{E}}(w)&\cdots&\frac{\partial^{2k}}{\partial w^{k}\partial \overline{w}^{k}}h_{\mathcal{E}}(w)
\end{pmatrix}.
\end{equation*}
Throughout the section, we denote by $D_{\alpha}^{*}$ and $D^{*}_{\alpha, E}$ the backward shift operators on weighted Dirichlet spaces $\mathcal{D}_{\alpha}$ and the spaces $\mathcal{D}_{\alpha, E}$, respectively, where $\alpha \in (0,1]$ and $E$ is a separable Hilbert space.  The reproducing kernel of $\mathcal{D}_{\alpha}$ is given by $K_{\alpha}(z,w)=\sum\limits_{i=0}^{\infty}\frac{z^{i}\overline{w}^{i}}{(i+1)^{\alpha}}, z,w\in\mathbb{D}$. In particular, $D_{1}^{*}$ is also denoted as $D^*$ and $\mathcal{D}_{1}$ as $\mathcal{D}$.

We will use the following result by M\"{u}ller in \cite{VM1988}, Theorem \ref{lem4}, to construct a defect operator corresponding to $T\in \mathcal{L}(\mathcal{H})$ with $\sum\limits_{s=1}^{\infty}\frac{\Vert T^{s}\Vert^{2}}{(s+1)^{\alpha}}\leq 1$ such that $T$ is unitary equivalent to the restriction of $D^{*}_{\alpha, E}$ on some invariant subspace:

\begin{lemma}\label{lemma4}
Let $T \in \mathcal{L}(\mathcal{H})$ and $\sum\limits_{s=1}^{\infty}\frac{\Vert T^{s}\Vert^{2}}{(s+1)^{\alpha}}\leq 1$ for some constant $\alpha\in(0,1]$. Then there exist a Hilbert space $E$ and an
$D^{*}_{\alpha, E}$-invariant subspace $\mathcal{N} \subseteq
\mathcal{D}_{\alpha, E}$ such that $T\sim_{u}D^{*}_{\alpha, E}|_{\mathcal{N}}$.
\end{lemma}

\begin{proof}
Let $b_{0}=1$, $c_{0}=1$, and for $k\geq1,$ set $b_{k}=\frac{1}{(k+1)^{\alpha}}$, and  $c_k=-\sum\limits_{j=0}^{k-1} \frac{c_j}{(k-j+1)^{\alpha}}$. Then for a fixed positive integer $l\geq1$,
$\sum\limits_{i=0}^{l}b_{i}c_{l-i}=0$. Obviously, $-b_{1}=c_{1}<0$. Without losing generality, assume that $-b_{i}\leq c_{i}<0$ for $1\leq i\leq k-1$. Then
by the assumption $\sum\limits_{s=1}^{\infty}\frac{\Vert T^{s}\Vert^{2}}{(s+1)^{\alpha}}\leq 1$ and Lemma 2.1 in \cite{VM1988},
we know that $0\geq c_{k}\geq -b_{k}$ for every $k\geq 1,$ which gives
\begin{equation*}\label{27}\small
\left\langle\left(\sum\limits_{k=0}^{\infty}c_{k}(T^{k})^{*}T^{k}\right)x,x\right\rangle
\geq\Vert x\Vert^{2}+\sum\limits_{k=1}^{\infty}c_{k}\Vert T^{k}\Vert^{2}\Vert x\Vert^{2}
\geq\Vert x\Vert^{2}-\sum\limits_{k=1}^{\infty}\frac{\Vert T^{k}\Vert^{2}}{(k+1)^{\alpha}}\Vert x\Vert^{2}
\geq0,\quad x\in\mathcal{H}.
\end{equation*}
Hence, $\sum\limits_{k=0}^{\infty}c_{k}(T^{*})^{k}T^{k}$ is a positive operator.

Next, we define an operator $V: \mathcal{H}\rightarrow \mathcal{N}$ by
$$Vx=\sum\limits_{n=0}^{\infty}\frac{z^{n}}{\Vert z^{n}\Vert_{\mathcal{D}_{\alpha}}^{2}}\otimes D_{T} T^{n} x,\quad x\in\mathcal{H},$$
where $\mathcal{N}=\overline{\text{ran }V}$ and $D_{T}=\left(\sum\limits_{k=0}^{\infty}c_{k}(T^{k})^{*}T^{k}\right)^{1/2}$. Since an orthonormal basis for the space $\mathcal{D}_{\alpha}$ is given by $\left\{{e}_{n}=\frac{z^{n}}{\sqrt{(n+1)^{\alpha}}}\right\}_{n\geq0}$, setting $E:=\overline{\text{ran }D_{T}}$, we have for every $x\in\mathcal{H},$
$$\small \Vert Vx\Vert^{2}=\sum\limits_{n=0}^{\infty}b_{n}\left\langle \left(\sum\limits_{k=0}^{\infty}c_{k}(T^{k})^{*}T^{k}\right)T^{n} x, T^{n} x\right\rangle
=\sum\limits_{l=0}^{\infty}\left(\Vert T^{l}x\Vert^{2}\sum\limits_{n+k=l}b_{n}c_{k}\right)
=\Vert x\Vert^{2},$$
and
$$\small VTx=\sum\limits_{n=0}^{\infty}\frac{z^{n}}{\Vert z^{n}\Vert_{\mathcal{D}_{\alpha}}^{2}}\otimes D_{T} T^{n+1} x
=\sum\limits_{n=1}^{\infty}\frac{z^{n-1}}{\Vert z^{n-1}\Vert_{\mathcal{D}_{\alpha}}^{2}}\otimes D_{T} T^{n} x
=D^{*}_{\alpha,E}\left(\sum\limits_{n=0}^{\infty}\frac{z^{n}}{\Vert z^{n}\Vert_{\mathcal{D}_{\alpha}}^{2}}\otimes D_{T} T^{n} x\right)
=D^{*}_{\alpha,E}Vx.$$
Hence, $T\sim_{u}D^{*}_{\alpha,E}|_{\mathcal{N}}$ where $\mathcal{N}\subset \mathcal{D}_{\alpha, E}$ is a $D_{\alpha,E}^*$-invariant subspace.
\end{proof}
We will denote as $$D_{T}=\left(\sum\limits_{k=0}^{\infty}c_{k}(T^{k})^{*}T^{k}\right)^{1/2},$$ the \emph{defect operator} of the operator $T$ with $\sum\limits_{s=1}^{\infty}\frac{\Vert T^{s}\Vert^{2}}{(s+1)^{\alpha}}\leq 1$. Then we have the following lemma from Lemma \ref{lemma4}, whose proof is omitted.

\begin{lemma}\label{lemma5}
Let $T\in\mathcal{B}_{1}(\mathbb{D})$ be an operator with $\sum\limits_{s=1}^{\infty}\frac{\Vert T^{s}\Vert^{2}}{(s+1)^{\alpha}}\leq 1$ for some constant $\alpha\in(0,1]$, and let $\gamma$ be a non-vanishing holomorphic cross-section of $\mathcal{E}_{T}$.
Then $\mathcal{E}_{D_{T}}=\left\{(w,x)\in\mathbb{D}\times\mathcal{E}_{D_{T}}(w):
x\in\mathcal{E}_{D_{T}}(w)\right\}$ is a Hermitian holomorphic line bundle over $\mathbb{D}$ with the natural projection $\pi, \pi(w,x)=w,$
where $\mathcal{E}_{D_{T}}(w)=\bigvee\{D_{T}\gamma(w)\}$.

Moreover, $\mathcal{E}_{T}\sim_{u}\mathcal{E}_{D^{*}_{\alpha}}\otimes \mathcal{E}_{D_{T}}$ and $D_{T}\gamma$ is a non-vanishing holomorphic cross-section of $\mathcal{E}_{D_{T}}$.
\end{lemma}

The ratio of the metrics on holomorphic line bundles, as derived in the previous section, serves as a similarity invariant for the backward shift operators on the Hardy and weighted Bergman spaces. However, whether this extends to weighted Dirichlet spaces remains unclear. In what follows, we give a sufficient condition for certain bounded linear operators to be similar to the backward shift operator $D_{\alpha}^{*}$ on weighted Dirichlet space $\mathcal{D}_{\alpha}$ in terms of the holomorphic jet bundle.

\begin{theorem}\label{thm1}
Let $T\in\mathcal{B}_{1}(\mathbb{D})$ and  $\sum\limits_{s=1}^{\infty}\frac{\Vert T^{s}\Vert^{2}}{(s+1)^{\alpha}}\leq 1$ for some constant $\alpha\in(0,1]$.
If there exist constants $c_{2}> c_{1}>0$ that satisfy
$$c_{1}\leq h_{\mathcal{E}_{D_{T}}}(w)<\text{trace}\;\mathcal{J}_{1}(h_{\mathcal{E}_{D_{T}}})(w) \leq c_{2},\quad w\in\mathbb{D},$$
for the metric $h_{\mathcal{E}_{D_{T}}}$ of $\mathcal{E}_{D_{T}}$ and $\mathcal{E}_{D_{T}}(w) \subset \mathbb{C}^n$ for some $n>0$, then $T\sim_{s}D^{*}_{\alpha}.$
\end{theorem}
\begin{proof}
Since $\sum\limits_{s=1}^{\infty}\frac{\Vert T^{s}\Vert^{2}}{(s+1)^{\alpha}}\leq 1$ for some constant $\alpha\in(0,1]$, by Lemmas \ref{lemma4} and \ref{lemma5}, $\mathcal{E}_{T}\sim_{u}\mathcal{E}_{D_{\alpha}^{*}}\otimes\mathcal{E}_{D_{T}}$ for the holomorphic line bundle $\mathcal{E}_{D_{T}}$ induced by the defect operator $D_{T}$.

Now, a holomorphic function $f\in\mathcal{D}_{\alpha}$ can be represented as a power series
$$f(z)=\sum\limits_{i=0}^{\infty}a_{i}z^{i},\quad\text{where }\quad\Vert f\Vert_{\mathcal{D}_{\alpha}}^{2}=\sum\limits_{i=0}^{\infty}(i+1)^{\alpha}|a_{i}|^{2}<\infty.$$ Since $\Vert zf\Vert_{\mathcal{D}_{\alpha}}^{2}=\sum\limits_{i=0}^{\infty}(i+2)^{\alpha}|a_{i}|^{2}\leq 2\sum\limits_{i=0}^{\infty}(i+1)^{\alpha}|a_{i}|^{2}<\infty$, $z\in\text{Mult}(\mathcal{D}_{\alpha})$. From \cite{L2022}, it is known that $\phi \in H^{\infty}(\mathbb{D})$ is in $\text{Mult}(\mathcal{D}_{\alpha})$ if and only if there is a constant $d>0$ such that
\begin{equation*}\label{06}
\int_{\mathbb{D}}\vert f(z)\vert^{2}\vert \phi^{'}(z)\vert^{2}(1-|z|^{2})^{1-\alpha}dA(z)\leq d\Vert f\Vert^{2}_{\mathcal{D}_{\alpha}},\quad f\in\mathcal{D}_{\alpha},
\end{equation*}
where $dA$ is the normalized area measure.
Setting $\phi(z)=z$, we have
\begin{equation}\label{15}
\int_{\mathbb{D}}\vert f(z)\vert^{2}(1-|z|^{2})^{1-\alpha}dA(z)\leq d\Vert f\Vert^{2}_{\mathcal{D}_{\alpha}},\quad f\in\mathcal{D}_{\alpha}.
\end{equation}

By assumption, the metric $\mathcal{J}_{1}(h_{\mathcal{E}_{D_{T}}})$ of the $1$-jet bundle $\mathcal{J}_{1}(\mathcal{E}_{D_{T}})$ satisfies
\begin{equation}\label{16}
c_{1}\leq h_{\mathcal{E}_{D_{T}}}(w)<\text{trace}\;\mathcal{J}_{1}(h_{\mathcal{E}_{D_{T}}})(w) \leq c_{2},\quad w\in\mathbb{D},
\end{equation}
for some constants $c_{2}> c_{1}>0$, where $h_{\mathcal{E}_{D_{T}}}$ denotes the metric obtained from a non-vanishing holomorphic cross-section $\gamma$ of $\mathcal{E}_{D_{T}}$. Furthermore, since $\mathcal{E}_{D_{T}}(w)\subset\mathbb{C}^{n}$ for some $n>0$, there is an orthonormal basis $\{e_{i}\}_{i=1}^{n}$ of the Hilbert space $E=\bigvee\limits_{w\in\mathbb{D}}\gamma(w)$ such that $\gamma(w)=\sum\limits_{i=1}^{n}\gamma_{i}(w)e_{i}$.
From (\ref{16}), it follows that
\begin{equation}\label{17}
c_{1}\leq h_{\mathcal{E}_{D_{T}}}(w)=\sum\limits_{i=1}^{n}|\gamma_{i}(w)|^{2}<\sum\limits_{i=1}^{n}|\gamma_{i}(w)|^{2}+\sum\limits_{i=1}^{n}|\gamma^{'}_{i}(w)|^{2}=
h_{\mathcal{E}_{D_{T}}}(w)+\frac{\partial^{2}}{\partial w\partial \overline{w}}h_{\mathcal{E}_{D_{T}}}(w)\leq c_{2},\quad w\in\mathbb{D}.
\end{equation}
Then by (\ref{15}) and (\ref{17}), we have for every $f \in \mathcal{D}_{\alpha}$,
\begin{equation}\label{25}
\int_{\mathbb{D}}\vert f(z)\vert^{2}\vert \gamma^{'}_{i}(z)\vert^{2}(1-|z|^{2})^{1-\alpha}dA(z)\leq
c_{2}\int_{\mathbb{D}}\vert f(z)\vert^{2}(1-|z|^{2})^{1-\alpha}dA(z)
\leq d c_{2}\Vert f\Vert^{2}_{\mathcal{D}_{\alpha}}.
\end{equation}
This shows that $\{\gamma_{1},\ldots,\gamma_{n}\}\subset\text{Mult}(\mathcal{D}_{\alpha})$ and that $\Vert M_{\gamma_{i}}\Vert=\Vert\gamma_{i}\Vert_{\text{Mult}(\mathcal{D}_{\alpha})}<\infty$ for $1\leq i\leq n.$

We next define an operator-valued function $F: \mathbb{D}\rightarrow \mathcal{L}(\mathbb{C}, E)$ by
$$F(w)(1)=\gamma(w).$$
Then for every $f\in\mathcal{D}_{\alpha}$,
$M_{F}f=f\gamma$,
and
$$\Vert M_{F}\Vert^{2}=\sup\limits_{\Vert f\Vert_{\mathcal{D}_{\alpha}}\neq 0}\frac{\Vert M_{F}f\Vert_{\mathcal{D}_{\alpha, E}}^{2}}{\Vert f\Vert^{2}_{\mathcal{D}_{\alpha}}}=
\sup\limits_{\Vert f\Vert_{\mathcal{D}_{\alpha}}\neq 0}\frac{\sum\limits_{i=1}^{n}\Vert f\gamma_{i}\Vert_{\mathcal{D}_{\alpha}}^{2}}{\Vert f\Vert^{2}_{\mathcal{D}_{\alpha}}}
\leq\sum\limits_{i=1}^{n}\Vert M_{\gamma_{i}}\Vert^{2}<\infty,$$
that is, $F\in \text{Mult}(\mathcal{D}_{\alpha},\mathcal{D}_{\alpha, E}).$
For $\gamma_{i}\in H^{\infty}(\mathbb{D}), 1 \leq i \leq n$, we let
$\gamma_{i}(z):=\sum\limits_{j=0}^{\infty}a_{i,j}z^{j}$ and
$\widetilde{\gamma}_{i}(z):=\sum\limits_{j=0}^{\infty}\overline{a_{i,j}}z^{j}$. Then
using $(\ref{17})$ and $(\ref{25})$, we deduce that for every $w \in \mathbb{D}$,
$|\widetilde{\gamma}_{i}(w)|^{2}, |\widetilde{\gamma}^{'}_{i}(w)|^{2}\leq c_{2}$ and
$$\int_{\mathbb{D}}\vert f(z)\vert^{2}\vert \widetilde{\gamma}^{'}_{i}(z)\vert^{2}(1-|z|^{2})^{1-\alpha}dA(z)\leq
c_{2}\int_{\mathbb{D}}\vert f(z)\vert^{2}(1-|z|^{2})^{1-\alpha}dA(z)
\leq d c_{2}\Vert f\Vert^{2}_{\mathcal{D}_{\alpha}},\,\, f\in\mathcal{D}_{\alpha}, 1\leq i\leq n.$$
Thus $\{\widetilde{\gamma}_{1},\ldots,\widetilde{\gamma}_{n}\}\subset\text{Mult}(\mathcal{D}_{\alpha})$ and $(F^{\#})^{*}\in\text{Mult}(\mathcal{D}_{\alpha, E},\mathcal{D}_{\alpha})$ for $F^{\#}(w)=F(\overline{w}).$
Since $1\in\text{Mult}(\mathcal{D}_{\alpha})$, $\gamma_{i}\in \text{Mult}(\mathcal{D}_{\alpha})$, and $\sum\limits_{i=1}^{n}|\gamma_{i}(w)|^{2}\geq c_{1}>0$ for every $w\in\mathbb{D}$, by Theorem \ref{tavan}, there are functions
$\{g_{1},\ldots,g_{n}\}\subset\text{Mult}(\mathcal{D}_{\alpha})$ such that
$\sum\limits_{i=1}^{n}g_{i}(w)\gamma_{i}(w)=1$ for every $ w\in\mathbb{D}.$

We now define a function $G:\mathbb{D}\rightarrow \mathcal{L}(E, \mathbb{C})$ by
$$G(w)h:=\sum\limits_{i=1}^{n}g_{i}(w)h_{i},$$
where $h=\sum\limits_{i=1}^{n}h_{i}e_{i}\in E.$
It is easy to see that $G\in\text{Mult}(\mathcal{D}_{\alpha,E}, \mathcal{D}_{\alpha})$ and $(G^{\#})^{*}\in\text{Mult}(\mathcal{D}_{\alpha},\mathcal{D}_{\alpha, E})$ for $G^{\#}(w)=G(\overline{w})$.
Hence, using Theorem \ref{book} and the reproducing kernel $K_{\alpha}(z,w)=\sum\limits_{n=0}^{\infty}\frac{z^{n}\overline{w}^{n}}{(n+1)^{\alpha}}$ of $\mathcal{D}_{\alpha}$,
$$M^*_{(F^{\#})^*}(K_{\alpha}(\cdot,\overline{w}) \otimes 1)=K_{\alpha}(\cdot,\overline{w}) \otimes F^{\#}(\overline{w})(1)=K_{\alpha}(\cdot,\overline{w}) \otimes F(w)(1)=K_{\alpha}(\cdot,\overline{w}) \otimes\gamma(w),$$
and therefore,
$$M^{*}_{(G^{\#})^{*}}M^{*}_{(F^{\#})^{*}}(K_{\alpha}(\cdot,\overline{w}) \otimes 1)=M^{*}_{(G^{\#})^{*}}(K_{\alpha}(\cdot,\overline{w}) \otimes F(w)(1))=K_{\alpha}(\cdot,\overline{w}) \otimes \sum\limits_{i=1}^{n}g_{i}(w)\gamma_{i}(w)=K_{\alpha}(\cdot,\overline{w})\otimes 1,$$
so that using the fact that
$\mathcal{D}_{\alpha}=\bigvee\limits_{w\in\mathbb{D}}K_{\alpha}(\cdot,\overline{w}),$ we conclude that
$M^{*}_{(G^{\#})^{*}}M^{*}_{(F^{\#})^{*}}\equiv I_{\mathcal{D}_{\alpha}}.$
Moreover, if we let $\mathcal{N}:=\bigvee\limits_{w\in\mathbb{D}}\{K_{\alpha}(\cdot,\overline{w})\otimes \gamma(w)\}$, then since the operator $M^{*}_{(F^{\#})^{*}}$ is left-invertible, $\text{ran }M^{*}_{(F^{\#})^{*}}=\mathcal{N}.$
Hence, the operator $X:=M^{*}_{(F^{\#})^{*}}:\mathcal{D}_{\alpha}\rightarrow \mathcal{N}$ is a bounded invertible operator satisfying
$$
X(\mathcal{E}_{D_{\alpha}^{*}}(w))=(\mathcal{E}_{D_{\alpha}^{*}}\otimes \mathcal{E}_{D_{T}})(w),\quad w\in\mathbb{D}.
$$
Finally, by considering $\mathcal{E}_{T}\sim_{u}\mathcal{E}_{D_{\alpha}^{*}}\otimes\mathcal{E}_{D_{T}}$, we conclude that $T \sim_{s} D^*_{\alpha}$.
\end{proof}

In Theorem \ref{thm1}, the ratio of the metrics of holomorphic line bundles $\mathcal{E}_{T}$ and $\mathcal{E}_{D^*_{\alpha}}$ is the metric of the holomorphic line bundle $\mathcal{E}_{D_{T}}$ induced by the defect operator $D_{T}$, that is, $h_{\mathcal{E}_{D_{T}}}=\frac{h_{\mathcal{E}_{T}}}{h_{D^*_{\alpha}}}$.
Letting $\mathcal{H}_{K}$ be a Hilbert space with a diagonal reproducing kernel ${K}(z,w)=\sum\limits_{m,n=0}^{\infty}a_{m,n}z^{m}\overline{w}^{n}$, which means that for all $n \neq m$, $a_{m,n}=0$, we obtain a similarity classification of holomorphic complex bundles and the similarity of backward shift operators on non-diagonal reproducing kernel Hilbert spaces.

\begin{corollary}\label{cor4.5}
Let $\mathcal{E}_{1}$ and $\mathcal{E}_{2}$ be Hermitian holomorphic line bundles over $\mathbb{D}$ where $\mathcal{E}_{1}\sim_{u}\mathcal{E}_{D_{\alpha}^{*}}\otimes\mathcal{E}_{2}$.
If there exist constants $c_{2}> c_{1}>0$ and a metric $h_{\mathcal{E}_{2}}$ of $\mathcal{E}_{2}$ such that $$c_{1}\leq h_{\mathcal{E}_{2}}(w)<\text{trace}\;\mathcal{J}_{1}(h_{\mathcal{E}_{2}})(w) \leq c_{2},$$ and $\mathcal{E}_{2}(w) \subset \mathbb{C}^n$ for some positive integer $n>0$, then $\mathcal{E}_{1}\sim_{s}\mathcal{E}_{D_{\alpha}^{*}}.$
\end{corollary}

\begin{example}
Let $T\in\mathcal{B}_{1}(\mathbb{D})$ and let $\mathcal{E}_{T}\sim_{u}\mathcal{E}_{D_{\alpha}^{*}}\otimes\mathcal{E}$ for some Hermitian holomorphic line bundle $\mathcal{E}$ over $\mathbb{D}$.
If the metric $h_{\mathcal{E}}$ of $\mathcal{E}$ satisfies $h_{\mathcal{E}}(w)=1+\sum\limits_{i=1}^{n}|\phi_{i}(w)|^{2}$ for some polynomials $\phi_{1}, \ldots, \phi_{n}$, and a positive integer $n$, then $\mathcal{E}_{T}\sim_{s}\mathcal{E}_{D_{\alpha}^{*}}$, that is, $T\sim_{s} D_{\alpha}^{*}$.
\end{example}

A necessary and sufficient condition for certain operators to be similar to the backward shift operators on weighted Dirichlet spaces in terms of holomorphic line bundles is given next as follows:

\begin{proposition}\label{pro4.7}
For an operator $T\in\mathcal{B}_{1}(\mathbb{D})$, if there exists a Hermitian holomorphic line bundle $\mathcal{E}$ over $\mathbb{D}$ such that $\mathcal{E}_{T}\sim_{u}\mathcal{E}_{D_{\alpha}^{*}}\otimes\mathcal{E}$, then $T\sim_{s} D_{\alpha}^{*}$ if and only if $\mathcal{E}_{D_{\alpha}^{*}}\otimes\mathcal{E}\sim_{s}\mathcal{E}_{D_{\alpha}^{*}}$.
\end{proposition}

If the reproducing kernel of a Hilbert space $\mathcal{H}_{K}$ is of the form ${K}(z,w)=\sum\limits_{m,n=0}^{\infty}a_{m,n}z^{m}\overline{w}^{n},$  where, for some $n \neq m$, $a_{m,n}\neq0$, then Theorem \ref{lem11} cannot be used to determine similarity. The following example is a good illustration of how one can use Theorem \ref{thm1} to conclude that two operators are similar:

\begin{example}
Let $M_{z}^{*}$ be the backward shift operator on a Hilbert space $\mathcal{H}_{K}$ with reproducing kernel
 \begin{eqnarray*}
K(z,w)
&=&2+z+\overline{w}+2z\overline{w}+
(z+\overline{w})\sum\limits_{n=1}^{\infty}\frac{2n+1}{n(n+1)}z^{n}\overline{w}^{n}\\
&&+(z^{2}+\overline{w}^{2})\sum\limits_{n=0}^{\infty}\frac{z^{n}\overline{w}^{n}}{n+1}
+\sum\limits_{n=0}^{\infty}\frac{4n^{2}+15n+13}{(n+1)(n+2)(n+3)}z^{n+2}\overline{w}^{n+2}, \quad z,w\in\mathbb{D}.
\end{eqnarray*}
Then $M_{z}^{*}\sim_{s}D^{*}.$
\end{example}
\begin{proof}
Since the reproducing kernel of the Dirichlet space $\mathcal{D}$ is given as $K_{1}(z,w)=\sum\limits_{i=0}^{\infty}\frac{z^{i}\overline{w}^{i}}{i+1}$, $z,w\in\mathbb{D}$,
$$K(z,w)=K_{1}(z,w)\cdot\left[1+(1+z+z^{2})(1+\overline{w}+\overline{w}^{2})\right].$$
Then there exists a Hermitian holomorphic line bundle $\mathcal{E}$ over $\mathbb{D}$ such that $\mathcal{E}_{M_{z}^{*}}=\mathcal{E}_{D^{*}}\otimes\mathcal{E}$. Moreover, a non-vanishing holomorphic cross-section $\gamma$ of $\mathcal{E}$ is such that $\Vert\gamma(w)\Vert^{2}=1+|1+w+w^{2}|^{2}$. Therefore, the metric $h_{\mathcal{E}}(w)=\Vert\gamma(w)\Vert^{2}$ of $\mathcal{E}$
satisfies $$1\leq h_{\mathcal{E}}(w)<\text{trace} \mathcal{J}_{1}(h_{\mathcal{E}})(w)\leq 19,\quad w\in\mathbb{D},$$
and the conclusion can be derived from Theorem \ref{thm1}.
\end{proof}

In Theorem \ref{thm1}, the condition that $\mathcal{E}_{D_{T}}(w)\subset \mathbb{C}^{n}$ for some integer $n>0$ was necessary. A natural question is whether the result also holds when $n=\infty$. The following example answers the question in the affirmative:

\begin{example}\label{ex4.91}
Let $M_{z}^{*}$ be the backward shift operator on a Hilbert space $\mathcal{H}_{{K}}$ with reproducing kernel $${K}(z,w)=\sum\limits_{n=0}^{\infty}\left(
\sum\limits_{i=0}^{n}\frac{1}{(i+1)(n-i+1)^{2}}\right)z^{n}\overline{w}^{n},\quad z,w\in\mathbb{D}.$$
Then there is a Hermitian holomorphic line bundle
$\mathcal{E}=\{(w, \gamma): w\in\mathbb{D}, \gamma(w)\in\mathcal{E}(w)\}$
with the natural projection $\pi, \pi(w, \gamma)=w,$ such that $\mathcal{E}_{M_{z}^{*}}=\mathcal{E}_{D^{*}}\otimes\mathcal{E}$
and $\Vert\gamma(w)\Vert^{2}=\sum\limits_{n=0}^{\infty}\frac{1}{(n+1)^{2}}|w|^{2n},$
where $\mathcal{E}(w)=\bigvee \gamma(w).$

Using the relationship between a weight sequence
and the reproducing kernel of a Hilbert space, we see that the weight sequence of $M_{z}^{*}$ is $\left\{\alpha_{n}=\sqrt{\frac{a_{n-1}}{a_{n}}}\right\}_{n\geq 1}$, where $a_{n}=\sum\limits_{i=0}^{n}\frac{1}{(i+1)(n-i+1)^{2}}.$
Since the weight sequence of  $D^{*}$ is $\left\{\beta_{n}=\sqrt{\frac{n+1}{n}}\right\}_{n\geq 1}$, we have that
$\prod\limits_{i=1}^{n}\beta_{i}
\Big/\prod\limits_{i=1}^{n}\alpha_{i}
=\sqrt{(n+1)a_{n}}\geq1.$

Now, to prove from Theorem \ref{lem11} that $M_{z}^{*}$ is similar to $D^{*}$, we only need to prove that $\lim\limits_{n\rightarrow\infty}(n+1)a_{n}=c$ for some constant $c.$ Since
$\lim\limits_{n\rightarrow\infty}\left[\sum\limits_{k=1}^{n}\frac{1}{k}-\ln n\right]= \curlyvee$ , where $\curlyvee$ stands for the Euler-Mascheroni constant,
$$\lim\limits_{n\rightarrow\infty}\frac{n+1}{(n+2)^{2}}\sum\limits_{k=1}^{n+1}\frac{1}{k}=0,~ \lim\limits_{n\rightarrow\infty}\frac{n+1}{(n+2)^{2}}\sum\limits_{k=1}^{n+1}\frac{1}{n+2-k}=0,
~\text{and}~
\lim\limits_{n\rightarrow\infty}\frac{n+1}{n+2}\sum\limits_{k=1}^{n+1}\frac{1}{(n+2-k)^{2}}=\frac{\pi^{2}}{6}.$$
Therefore,
 \begin{eqnarray*}
\lim\limits_{n\rightarrow\infty}(n+1)a_{n}
&=&\lim\limits_{n\rightarrow\infty}(n+1)\left(
\sum\limits_{k=1}^{n+1}\frac{1}{k(n+2-k)^{2}}\right)\\
&=&\lim\limits_{n\rightarrow\infty}(n+1)\sum\limits_{k=1}^{n+1}\left[\frac{1}{(n+2)^{2}}\frac{1}{k}+
\frac{1}{(n+2)^{2}}\frac{1}{n+2-k}+\frac{1}{n+2}\frac{1}{(n+2-k)^{2}}\right]\\
&=&\frac{\pi^{2}}{6}.
\end{eqnarray*}
\end{example}

The following is another major theorem of this section. We provide a criterion for the similarity of backward shift operators on weighted Dirichlet spaces when $n=\infty$:

\begin{theorem}\label{thm2}
Let $T\in\mathcal{B}_{1}(\mathbb{D})$ and  suppose that $\sum\limits_{s=1}^{\infty}\frac{\Vert T^{s}\Vert^{2}}{(s+1)^{\alpha}}\leq 1$ for some constant $\alpha\in(0,1]$.
If there is a non-vanishing holomorphic cross-section $\gamma$ of $\mathcal{E}_{D_{T}}$ such that for some $\delta >0$,
\begin{equation}\label{equ4.10}
\delta\leq\Vert\gamma(w)\Vert_{E}^{2}\leq
\sum\limits_{i=1}^{\infty}\Vert\gamma_{i}\Vert^{2}_{\text{Mult}(\mathcal{D}_{\alpha})}<\infty,
\quad w\in\mathbb{D},
\end{equation}
where $\gamma(w)=\sum\limits_{i=1}^{\infty}\gamma_{i}(w)e_{i}$, $E=\bigvee\limits_{w\in\mathbb{D}}\gamma(w),$ and $\{e_{i}\}_{i=1}^{\infty}$ is an orthonormal basis of $E$, then $T\sim_{s}D_{\alpha}^{*}.$
\end{theorem}
\begin{proof}
Denoting by $K_{\alpha}(z,w)=\sum\limits_{i=0}^{\infty}\frac{z^{i}\overline{w}^{i}}{(i+1)^{\alpha}}, z,w\in\mathbb{D},$ the reproducing kernel for the space $\mathcal{D}_{\alpha}$, we have for every $f\in\mathcal{D}_{\alpha}$ and $\phi\in\text{Mult}(\mathcal{D}_{\alpha})$,
$$\langle f, M_{\phi}^{*}K_{\alpha}(\cdot,w)\rangle=\langle M_{\phi}f, K_{\alpha}(\cdot,w)\rangle=\phi(w)f(w)=\langle f, \overline{\phi(w)}K_{\alpha}(\cdot,w)\rangle,$$
and therefore,
$\Vert\phi\Vert_{\infty}\leq\Vert M_{\phi}\Vert=\Vert \phi \Vert_{\text{Mult}(\mathcal{D}_{\alpha})}.$
Let $\gamma$ be a non-vanishing holomorphic cross-section of $\mathcal{E}_{D_{T}}$ satisfying (\ref{equ4.10}).
If we assume without loss of generality that $\sum\limits_{i=1}^{\infty}\Vert\gamma_{i}\Vert^{2}_{\text{Mult}(\mathcal{D}_{\alpha})}\leq 1$, then
$$\Vert\gamma(w)\Vert_{E}^{2}=\sum\limits_{i=1}^{\infty}|\gamma_{i}(w)|^{2}\leq\sum\limits_{i=1}^{\infty}\Vert \gamma_{i}\Vert_{\infty}^{2}\leq\sum\limits_{i=1}^{\infty}\Vert\gamma_{i}\Vert^{2}_{\text{Mult}(\mathcal{D}_{\alpha})}\leq 1,\quad w\in\mathbb{D}.$$

We next define an operator-valued function $F: \mathbb{D}\rightarrow\mathcal{L}(\mathbb{C}, E)$ by
$$F(w)(1)=\gamma(w).$$ Then $F\in H^{\infty}_{\mathbb{C}\rightarrow E}(\mathbb{D})$ and the multiplication operator
$M_{F}: \mathcal{D}_{\alpha}\rightarrow\mathcal{D}_{\alpha, E}$
 satisfies
$$(M_{F}f)(z)=F(z)f(z)=f(z)\otimes \gamma(z).$$
Moreover, for every $f\in\mathcal{D}_{\alpha}$,
$$\Vert f\otimes \gamma\Vert_{\mathcal{D}_{\alpha}\otimes E}^{2}=\sum\limits_{i=1}^{\infty}\Vert f\gamma_{i}\Vert_{\mathcal{D}_{\alpha}}^{2}=\sum\limits_{i=1}^{\infty}\Vert M_{\gamma_{i}}f\Vert_{\mathcal{D}_{\alpha}}^{2}\leq
\Vert f\Vert_{\mathcal{D}_{\alpha}}^{2}\sum\limits_{i=1}^{\infty}\Vert M_{\gamma_{i}}\Vert^{2}
=\Vert f\Vert_{\mathcal{D}_{\alpha}}^{2}\sum\limits_{i=1}^{\infty}\Vert\gamma_{i}\Vert^{2}_{\text{Mult}(\mathcal{D}_{\alpha})}\leq\Vert f\Vert_{\mathcal{D}_{\alpha}}^{2}.$$
It follows that $F\in \text{Mult}(\mathcal{D}_{\alpha},\mathcal{D}_{\alpha,E})$.

Now let $F^{\#}(w):=F(\overline{w})$. Obviously, $(F^{\#})^{*}\in\text{Mult}(\mathcal{D}_{\alpha, E},\mathcal{D}_{\alpha})$. If we denote by $M_{\gamma}^{C}$ the column operator from $\mathcal{D}_{\alpha}$ to $\bigoplus\limits_{1}^{\infty}\mathcal{D}_{\alpha}$
defined by
$$M_{\gamma}^{C}(f)=(\gamma_{1}f,\gamma_{2}f,\cdots)^{T}=f\otimes\gamma,\quad f\in\mathcal{D}_{\alpha},$$
then $M_{F}=M_{\gamma}^{C}$ and
$$\Vert M_{\gamma}^{C}\Vert=\Vert M_{F}\Vert=\sup\limits_{\Vert f\Vert_{\mathcal{D}_{\alpha}}\neq0}\frac{\Vert M_{F}f\Vert_{\mathcal{D}_{\alpha, E}}}{\Vert f\Vert_{\mathcal{D}_{\alpha}}}
=\sup\limits_{\Vert f\Vert_{\mathcal{D}_{\alpha}}\neq0}\frac{\Vert f\otimes \gamma\Vert_{\mathcal{D}_{\alpha, E}}}{\Vert f\Vert_{\mathcal{D}_{\alpha}}}\leq 1.$$
Since $\Vert\gamma(w)\Vert_{E}^{2}=\sum\limits_{i=1}^{\infty}|\gamma_{i}(w)|^{2}\geq \delta>0$, by Theorem \ref{tavan}, we obtain functions $\{g_{i}\}_{i=1}^{\infty}\subset\text{Mult}(\mathcal{D}_{\alpha})$
such that $\sum\limits_{i=1}^{\infty}g_{i}(w)\gamma_{i}(w)=1$ for any $w\in\mathbb{D}.$ Moreover, the column operator $M_{G}^{C}:\mathcal{D}_{\alpha}\rightarrow\bigoplus\limits_{1}^{\infty}\mathcal{D}_{\alpha}$ defined by $$M_{G}^{C}(f)=(g_{1}f,g_{2}f,\cdots)^{T},\quad f\in\mathcal{D}_{\alpha},$$ is bounded and from Theorem \ref{rowcolumn}, we know that the row operator
$M_{G}^{R}:\bigoplus\limits_{1}^{\infty}\mathcal{D}_{\alpha}\rightarrow\mathcal{D}_{\alpha}$ is bounded as well.

Finally, we define an operator-valued function $\widetilde{G}: \mathbb{D}\rightarrow \mathcal{L}(E,\mathbb{C})$ by
$$\widetilde{G}(w)(h)=\sum\limits_{i=1}^{\infty}g_{i}(w)h_{i}, \quad h=\sum\limits_{i=1}^{\infty}h_{i}e_{i}\in E.$$
Then the operator
$M_{\widetilde{G}}: \mathcal{D}_{\alpha,E}\rightarrow \mathcal{D}_{\alpha}$
defined by $$M_{\widetilde{G}}f=\widetilde{G}f, \quad f\in \mathcal{D}_{\alpha,E},$$
is such that
$\Vert M_{\widetilde{G}}\Vert\leq\Vert M_{G}^{R}\Vert<\infty,$
so that $M_{\widetilde{G}}$ is bounded and $\widetilde{G}\in\text{Mult}(\mathcal{D}_{\alpha, E},\mathcal{D}_{\alpha})$.
Since $M_{G}^{C}$ is bounded, $(\widetilde{G}^{\#})^{*}\in\text{Mult}(\mathcal{D}_{\alpha},\mathcal{D}_{\alpha, E}),$ where $\widetilde{G}^{\#}(w)=\widetilde{G}(\overline{w}).$
Proceeding similarly as in the proof of Theorem \ref{thm1}, we obtain an invertible operator $X$ such that
$X(\mathcal{E}_{D_{\alpha}^{*}}(w))=(\mathcal{E}_{D_{\alpha}^{*}}\otimes \mathcal{E}_{D_{T}})(w),$ and since $\mathcal{E}_{T}\sim_{u}\mathcal{E}_{D_{\alpha}^{*}}\otimes\mathcal{E}_{D_{T}}$,
$T\sim_{s}D_{\alpha}^{*}$ as claimed.
\end{proof}

\begin{remark}
The condition (\ref{equ4.10}) in Theorem \ref{thm2} is a sufficient but not necessary condition, as illustrated in Example \ref{ex4.91}. In Example \ref{ex4.91}, $M_{z}^{*}$ is similar to $D^{*}$. From the holomorphic frame $\gamma$ of the holomorphic vector bundle $\mathcal{E}$, we obtain $\gamma(w)=\sum\limits_{i=1}^{\infty}\gamma_{i}(w)e_{i}$, where $\gamma_{i}(w)=\frac{1}{i}w^{i-1}$, and hence
$$\sum\limits_{i=1}^{\infty}\Vert\gamma_{i}\Vert^{2}_{\text{Mult}(\mathcal{D})}\geq\sum\limits_{i=1}^{\infty}\Vert \gamma_{i}\Vert_{\mathcal{D}}^{2}=\sum\limits_{i=1}^{\infty}\frac{1}{i}=\infty.$$
\end{remark}

The condition $\sum\limits_{s=1}^{\infty}\frac{\Vert T^{s}\Vert^{2}}{(s+1)^{\alpha}}\leq 1$ in Theorem \ref{thm2} ensures that the Hermitian holomorphic line bundle $\mathcal{E}_{T}$ acquires a tensor product structure. A natural question then arises: which types of operators correspond to Hermitian holomorphic complex bundles with a tensor product structure in the sense of unitary equivalence? We will give the precise answer, distinct from the condition $\sum\limits_{s=1}^{\infty}\frac{\Vert T^{s}\Vert^{2}}{s+1}\leq 1$, through the following example:

\begin{example}\label{ex4.10}
Let $M_{z}^{*}$ be the backward shift operator on a Hilbert space $\mathcal{H}_{K}$ with reproducing kernel $K(z,w)=\sum\limits_{n=0}^{\infty}a_{n}z^{n}\overline{w}^{n}$, $z,w\in\mathbb{D}$, where $a_{n}=\sum\limits_{\mbox{\tiny$\begin{array}{c}
i+j^{5}=n\\ i,j\in\mathbb{N}\end{array}$}}\frac{1}{(i+1)(j+1)^{8}}$ for $n\in\mathbb{N}$.
In this case, $\left\{\sqrt{a_{n}}z^{n}\right\}_{n\geq 0}$ is an orthonormal basis for $\mathcal{H}_{K}$ and $\left\{\sqrt{\frac{a_{n-1}}{a_{n}}}\right\}_{n\geq 1}$ is the weight sequence of $M^*_z$. It can be computed that
\begin{equation}\label{equ1}
K(z,w)=\sum\limits_{n=0}^{\infty}\left(\sum\limits_{\mbox{\tiny$\begin{array}{c}
i+j^{5}=n\\ i,j\in\mathbb{N}\end{array}$}}\frac{1}{(i+1)(j+1)^{8}}\right)z^{n}\overline{w}^{n}=\left(\sum\limits_{i=0}^{\infty}\frac{z^{i}\overline{w}^{i}}{i+1}\right)
\left(\sum\limits_{j=0}^{\infty}\frac{z^{j^{5}}\overline{w}^{j^{5}}}{(j+1)^{8}}\right).
\end{equation}
If $M_{z}^{*}$ satisfies $\sum\limits_{s=1}^{\infty}\frac{\Vert M_{z}^{*s}\Vert^{2}}{s+1}\leq 1,$ then
$$\left\langle\left(I-\sum\limits_{s=1}^{\infty}\frac{ M_{z}^{s}M_{z}^{*s}}{s+1}\right)x,x\right\rangle
=\Vert x\Vert^{2}-\sum\limits_{s=1}^{\infty}\frac{\Vert M_{z}^{*s}x\Vert^{2}}{s+1}
\geq\left(1-\sum\limits_{s=1}^{\infty}\frac{\Vert M_{z}^{*s}\Vert^{2}}{s+1}\right)\Vert x\Vert^{2}
\geq0, \quad x\in\mathcal{H}_{K}.$$
It follows that $\sum\limits_{s=1}^{\infty}\frac{M_{z}^{s}M_{z}^{*s}}{s+1}\leq I$, so that $a_{n}-\frac{1}{2}a_{n-1}-\cdots-\frac{1}{n+1}a_{0}\geq 0$ for every positive integer $n$.
However, for $a_{2}=\frac{515}{1536}$, $a_{1}=\frac{129}{256}$ and $a_{0}=1$, $a_{2}-\frac{1}{2}a_{1}-\frac{1}{3}a_{0}=\frac{515}{1536}-\frac{1}{2}\times\frac{129}{256}-\frac{1}{3}=-\frac{1}{4}<0$. This means that $M_{z}^{*}$ does not satisfy $\sum\limits_{s=1}^{\infty}\frac{\Vert M_{z}^{*s}\Vert^{2}}{s+1}\leq 1$.

If we set $\mathcal{E}$ to be a Hermitian holomorphic line bundle over $\mathbb{D}$ with a non-vanishing holomorphic cross-section $\gamma$ satisfying
$\Vert\gamma(w)\Vert^{2}=\sum\limits_{j=0}^{\infty}\frac{|w|^{2j^{5}}}{(j+1)^{8}}$, then from (\ref{equ1}),  $\mathcal{E}_{M_{z}^{*}}\sim_{u}\mathcal{E}_{D^{*}}\otimes\mathcal{E}$ for the backward shift operator $D^{*}$ on the Dirichlet space $\mathcal{D}$.
Let $E=\bigvee\limits_{w\in\mathbb{D}}\gamma(w)$ and let $\{e_{k}\}_{k=0}^{\infty}$ be an orthonormal basis of $E$,
so that $\gamma(w)=\sum\limits_{k=0}^{\infty}\gamma_{k}(w)e_{k},$ where $\gamma_{k}(w)=\frac{w^{k^{5}}}{(k+1)^{4}}$.

Since for $f\in\mathcal{D}$, $\Vert f\Vert_{\mathcal{D}}^{2}=\sum\limits_{n=0}^{\infty}(n+1)|\hat{f}(n)|^{2}<\infty$, it is easily seen that
$\gamma_{k}(z)f(z)=\sum\limits_{n=0}^{\infty}\frac{\hat{f}(n)}{(k+1)^{4}}z^{n+k^{5}}, k \geq 0$. It follows that
 \begin{eqnarray*}
\Vert \gamma_{k}f\Vert_{\mathcal{D}}^{2}
&=&\sum\limits_{n=0}^{\infty}(n+k^{5}+1)\left|\frac{\hat{f}(n)}{(k+1)^{4}}\right|^{2}\\
&=&\frac{1}{(k+1)^{8}}\left[\sum\limits_{n=0}^{\infty}(n+1)|\hat{f}(n)|^{2}+\sum\limits_{n=0}^{\infty}
k^{5}|\hat{f}(n)|^{2}\right]\\
&\leq&\left[\frac{1}{(k+1)^{8}}+\frac{1}{(k+1)^{3}}\right]\Vert f\Vert_{\mathcal{D}}^{2}.
\end{eqnarray*}
Therefore, $\Vert\gamma_{k}\Vert^{2}_{\text{Mult}(\mathcal{D})} \leq\frac{1}{(k+1)^{8}}+\frac{1}{(k+1)^{3}}\leq\frac{2}{(k+1)^{2}}$, and hence,  $$1\leq\Vert\gamma(w)\Vert_{E}^{2}=\sum\limits_{k=0}^{\infty}|\gamma_{k}(w)|^{2}\leq
\sum\limits_{k=0}^{\infty}\Vert\gamma_{k}\Vert_{\infty}^{2}\leq
\sum\limits_{k=0}^{\infty}\Vert\gamma_{k}\Vert^{2}_{\text{Mult}(\mathcal{D})}\leq \frac{\pi^{2}}{3},\quad w\in\mathbb{D}.$$
Using the proof of Theorem \ref{thm2}, we conclude that $M_{z}^{*}\sim_{s}D^{*}$.
\end{example}

In Example \ref{ex4.10}, it was shown that the weighted backward shift operator $M_{z}^{*}$ does not satisfy the condition $\sum\limits_{s=1}^{\infty}\frac{\Vert M_{z}^{*s}\Vert^{2}}{s+1}\leq 1$. Furthermore, it is not straightforward to verify that $M_{z}^{*}$ is similar to $D^{*}$ using the result of Shields, Theorem \ref{lem11}. We now consider an example of a Hilbert space with a non-diagonal reproducing kernel. The proof, being similar to that of Example \ref{ex4.10}, is omitted.

\begin{example}\label{ex4.9}
Let $M_{z}^{*}$ be the backward shift operator on a Hilbert space $\mathcal{H}_{K}$ with reproducing kernel
$${K}(z,w)=1+\frac{5}{8}z\overline{w}+
(z+\overline{w})\sum\limits_{n=1}^{\infty}\frac{1}{8n}z^{n}\overline{w}^{n}
+\sum\limits_{n=2}^{\infty}\left(\sum\limits_{i=1}^{n+1}\frac{1}{i(n+2-i)^{3}}+
\frac{1}{8(n-1)}\right)z^{n}\overline{w}^{n},\quad z,w\in\mathbb{D}.$$
Then  $M_{z}^{*}\sim_{s}D^{*}.$
\end{example}

We close this section with a proposition stating that to a Hermitian holomorphic complex bundle with a tensor product structure satisfying some given conditions, there corresponds a Cowen-Douglas operator. This result establishes the operator realization of the tensor product of holomorphic complex bundles.

\begin{proposition}\label{pro4.13}
Let $\mathcal{E}$ be a Hermitian holomorphic line bundle over $\mathbb{D}$ and let $\gamma$ be a non-vanishing holomorphic cross-section of $\mathcal{E}$. Suppose that
$E=\bigvee\limits_{w\in\mathbb{D}}\gamma(w)$ and $\{e_{i}\}_{i=0}^{\infty}$ is an orthonormal basis for $E$.
If
$$
\delta\leq\Vert\gamma(w)\Vert_{E}^{2}\leq
\sum\limits_{i=0}^{\infty}\Vert\gamma_{i}\Vert^{2}_{\text{Mult}(\mathcal{D}_{\alpha})}<\infty,\quad w\in\mathbb{D},
$$
for $\gamma(w)=\sum\limits_{i=0}^{\infty}\gamma_{i}(w)e_{i}$ and some constant $\delta >0$, then there exists an operator $T\in\mathcal{B}_{1}(\mathbb{D})$ such that $\mathcal{E}_{T}=\mathcal{E}_{D_{\alpha}^{*}}\otimes\mathcal{E}$ and  $T\sim_{s}D_{\alpha}^{*}.$
\end{proposition}

\section*{Concluding remarks}

Recall that the Cowen-Douglas conjecture states that for operators $T$ and $S$ in $\mathcal{B}_{1}(\mathbb{D})$ having $\overline{\mathbb{D}}$ as a $K$-spectral set, $T$ is similar to $S$ if and only if $\lim\limits_{\vert w\vert\rightarrow 1}\mathcal{K}_{T}(w)/\mathcal{K}_{S}(w)=1$. In \cite{CM2,CM}, Clark and Misra constructed examples to show that the Cowen-Douglas conjecture is false and claimed that the ratio of the metrics can be a better indication of similarity than the ratio of the curvatures.
In particular, several examples are provided in Theorem 1 of \cite{CM} to support this claim.
Theorem \ref{t1} and Proposition \ref{3.6} in this note offer further evidence.

Every Cowen-Douglas operator is unitarily equivalent to the adjoint of a multiplication operator on a reproducing kernel Hilbert space consisting of holomorphic functions.
For the Hardy and weighted Bergman spaces, Theorem 1 in \cite{CM}, Theorem 3.8 in \cite{UM1990}, Theorem 3.4 in \cite{JJ2022}, and Theorem \ref{t1} and Proposition \ref{3.6} of this paper suggest that, under certain conditions, the ratio of the metrics is a similarity invariant of the backward shift operator on these spaces.
In particular, for an $n$-hypercontraction $T\in\mathcal{B}_{1}(\mathbb{D})$, from Theorem 0.1 in \cite{KT}, Theorem 2.4 in \cite{DKT}, Theorem 2.3 in \cite{HJK}, and Theorem \ref{t1} of this paper, we conclude that $T$ is similar to $S_{n}^{*}$ on the reproducing kernel Hilbert space $\mathcal{M}_n$ if and only if $a_{w}:=h_{\mathcal{E}_{T}}(w)/h_{\mathcal{E}_{S_{n}^{*}}}(w)$ is bounded and bounded away from $0$ if and only if the curvature $\mathcal{K}_{\mathcal{E}_{D_{T}}}$ of the holomorphic line bundle $\mathcal{E}_{D_{T}}$
induced by the defect operator $D_{T}$ satisfies $\mathcal{K}_{\mathcal{E}_{D_{T}}}=-\frac{\partial^{2}}{\partial w \partial \overline{w}}\varphi (w)$ for some bounded subharmonic function $\varphi$ defined on $\mathbb{D}$.

Clark and Misra constructed contractive backward shift operators in 2.5 of \cite{CM}, such that the ratio of the metrics of the holomorphic line bundles associated with the operators is bounded and bounded away from $0$ that are not similar to each other. This indicates that the ratio of the metrics is not a similarity invariant for all backward shift operators on Hilbert spaces with diagonal reproducing kernels. One can consider a follow-up question.

\begin{question}\label{que1}
Is the ratio of the metrics of holomorphic line bundles a similarity invariant for the backward shift operator on the weighted Dirichlet space?
\end{question}

A similarity invariant for backward shift operators on weighted Dirichlet spaces, unfortunately, cannot be obtained from previous statements about Hardy and weighted Bergman spaces. The main reason for this has to do with the corresponding multiplier algebra; as is well-known, the multiplier algebra of weighted Dirichlet spaces is not equal to $H^{\infty}(\mathbb{D})$. For the similarity of backward shift operators on weighted Dirichlet spaces, assuming that the ratio of the metrics is bounded and bounded away from $0$, we have introduced new perspective on the problem by considering holomorphic jet bundles and provided two sufficient conditions in Theorems \ref{thm1} and \ref{thm2}. Meanwhile, as can be seen through Examples \ref{ex4.91} and \ref{ex4.10}, it is challenging to construct a backward shift operator that satisfies the following two conditions: (1) the ratio of the metric of the holomorphic line bundle induced by the operator to the metric of the holomorphic line bundle associated with the backward shift operator on the Dirichlet space is bounded (or unbounded or tends to $0$ as $|w|\rightarrow 1$), (2) the operator is not similar to the backward shift operator on the Dirichlet space.

\section*{Acknowledgment}

We gratefully acknowledge the valuable help of Y. H. Zhang and J. M. Yang on Example \ref{ex4.91}.


\begin{thebibliography}{99}
\bibitem{Agler}J. Agler, \emph{The Arveson extension theorem and coanalytic models}, Integral Equ. Oper. Theory, $\mathbf{5}$ (1982), 608-631.

\bibitem{Agler2}J. Agler, \emph{Hypercontractions and subnormality,} J. Operator Theory, $\mathbf{13}$ (1985), 203-217.

\bibitem{AM2002}{J. Agler and J. E. McCarthy}, \emph{Pick Interpolation and Hilbert Function Spaces}, Vol. 44 of Graduate Studies in Mathematics, American Mathematical Society, Providence, RI, 2002.

\bibitem{AEM2002}{C.-G. Ambrozie, M. Engli\v{s} and V. M\"{u}ller}, \emph{Operator tuples and analytic models over general domains in $\mathbb{C}^{n}$}, J. Operator Theory, $\mathbf{47}$ (2002), no. 2, 287-302.

\bibitem{AM2003}{J. Arazy and M. Engli\v{s}}, \emph{Analytic models for commuting operator tuples on bounded symmetric domains}, Trans. Amer. Math. Soc. $\mathbf{355}$ (2003), no. 2, 837-864.

\bibitem{AA1992}{A. Athavale}, \emph{Model theory on the unit ball in $\mathbb{C}^{m}$}, J. Operator Theory,  $\mathbf{27}$ (1992), no. 2, 347-358.

\bibitem{BTV1997}{J. A. Ball, T. T. Trent and V. Vinnikov}, \emph{Interpolation and commutant lifting for multipliers on reproducing kernel Hilbert spaces} in Operator Theory and Analysis, Oper. Theory Adv. Appl., 122, Birkh\"{a}user, Basel, 2001, 89-138.

\bibitem{CFJ}{Y. Cao, J. S. Fang and C. L. Jiang}, \emph{K-groups of Banach algebras and strongly irreducible decomposition of operators}, J. Operator Theory, $\mathbf{48}$ (2002), no. 1, 235-253.

\bibitem{CL1962}{L. Carleson}, \emph{Interpolations by bounded analytic functions and the corona problem}, Ann. of Math. $\mathbf{76}$ (1962), no. 2, 547-559.

\bibitem{CM2}{D. N. Clark and G. Misra}, \emph{On curvature and similarity}, Michigan Math. J. $\mathbf{30}$ (1983), no. 3, 361-367.

\bibitem{CM}{D. N. Clark and G. Misra}, \emph{On weighted shifts, curvature, and similarity}, J. Lond. Math. Soc. $\mathbf{31}$ (1985), no. 2, 357-368.

\bibitem{CM1978}{D. N. Clark and J. H. Morrel}, \emph{On Toeplitz operators and similarity}, Amer. J. Math. $\mathbf{100}$ (1978), no. 5, 973-986.

\bibitem{CD1}{M. J. Cowen and R. G. Douglas}, \emph{Complex geometry and operator theory}, Acta. Math.  $\mathbf{141}$ (1978), no. 1, 187-261.

\bibitem{CD2}{M. J. Cowen and R. G. Douglas}, \emph{Equivalence of connections}, Adv. Math.  $\mathbf{56}$ (1985), no. 1, 39-91.

\bibitem{CS2}{R. E. Curto and N. Salinas}, \emph{Generalized Bergman kernels and the Cowen-Douglas theory}, Am. J. Math. $\mathbf{106}$ (1984), no. 2, 447-488.

\bibitem{DKT}{R. G. Douglas, H. Kwon and S. Treil}, \emph{Similarity of $n$-hypercontractions and backward Bergman shifts}, J. Lond. Math. Soc. $\mathbf{88}$ (2013), no. 3, 637-648.

\bibitem{EKMR}{O. El-Fallah, K. Kellay, J. Mashreghi and T. Ransford}, \emph{A primer on the Dirichlet space} in Cambridge Tracts in Mathematics, 203, Cambridge University Press, Cambridge, 2014.

\bibitem{CF1963}{C. Foia\c{s}}, \emph{A remark on the universal model for contractions of G. C. Rota}, Com. Acad. R. P. Rom\^{1}ne, $\mathbf{13}$ (1963), 349-352.

\bibitem{PAF1968}{P. A. Fuhrmann}, \emph{On the corona theorem and its application to spectral problems in Hilbert space}, Trans. Amer. Math. Soc. $\mathbf{132}$ (1968), 55-66.

\bibitem{GJ1985}{J. B. Garnett}, \emph{Bounded analytic functions}, 1st ed., Graduate Texts in Mathematics, vol. 236, Springer, New York, 2007.

\bibitem{GKP2022}{S. Ghara, S. Kumar and P. Pramanick}, \emph{$\mathbb{K}$-homogeneous tuple of operators on bounded symmetric domains}, Isr. J. Math. $\mathbf{247}$ (2022), no. 1, 331-360.

\bibitem{GG1973}{M. Golubitsky and V. Guillemin}, \emph{Stable mappings and their singularities}, Springer-Verlag, New York, 1973.

\bibitem{LH1967}{L. H\"{o}rmander}, \emph{Generators for some rings of analytic functions}, Bull. Amer. Math. Soc. $\mathbf{73}$ (1967), 943-949.

\bibitem{HJK}{Y. L. Hou, K. Ji and H. Kwon}, \emph{The trace of the curvature determines similarity}, Studia Math. $\mathbf{236}$ (2017), no. 2, 193-200.

\bibitem{JJ2022}{K. Ji and S. S. Ji}, \emph{The metrics of Hermitian holomorphic vector bundles and the similarity of Cowen-Douglas operators}, Indian J. Pure Appl. Math. $\mathbf{53}$ (2022), 736-749.

\bibitem{JJKX}{K. Ji, S. S. Ji, H. Kwon and J. Xu}, \emph{The Cowen-Douglas Theory for operator tuples and similarity}, https://arxiv.org/abs/2210.00209v2.

\bibitem{JJDG}{K. Ji, C. L. Jiang, D. K. Keshari and G. Misra},  \emph{Rigidity of the flag structure for a class of Cowen-Douglas operators}, J. Funct. Anal. $\mathbf{272}$ (2017), no. 7, 2899-2932.

\bibitem{JKSX2022}{K. Ji, H. Kwon, J. Sarkar and J. Xu},  \emph{A subclass of the Cowen-Douglas class and similarity}, Math. Nachr. $\mathbf{295}$ (2022), no. 11, 2197-2222.


\bibitem{JS2019}{K. Ji and J. Sarkar},  \emph{Similarity of quotient Hilbert modules in the Cowen-Douglas class}, Eur. J. Math. $\mathbf{5}$ (2019), no. 4, 1331-1351.


\bibitem{JX2022}{S. S. Ji and J. Xu}, \emph{On the trace of curvature of irreducible Cowen-Douglas operators}, J. Math. Anal. Appl. $\mathbf{512}$ (2022), 126181.

\bibitem{J}{C. L. Jiang}, \emph{Similarity classfication of Cowen-Douglas operators},  Can. J. Math. $\mathbf{56}$ (2004), no. 4, 742-775.

\bibitem{JGJ}{C. L. Jiang, X. Z. Guo and K. Ji,} \emph{K-group and similarity classification of operators}, J. Funct. Anal. $\mathbf{225}$ (2005), no. 1, 167-192.

\bibitem{CK}{C. L. Jiang and K. Ji,} \emph{Similarity classification of holomorphic curves}, Adv. Math. $\mathbf{215}$ (2007), 446-468.

\bibitem{CL}{C. L. Jiang and Y. C. Li,} \emph{The commutant and similarity invariant of analytic Toeplitz operators on Bergman space}, Sci. China Ser. A $\mathbf{50}$ (2007),  no. 5, 651-664.

\bibitem{DKK2014}{D. K. Keshari}, \emph{Trace formulae for curvature of jet bundles over planar domains}, Complex Anal. Oper. Theory, $\mathbf{8}$ (2014), no. 8, 1723-1740.

\bibitem{TTT2013}{B. Kidane and T. T. Trent}, \emph{The corona theorem for multiplier algebras on weighted Dirichlet spaces}, Rocky Mt. J. Math. $\mathbf{43}$ (2013), no. 1, 241-271.

\bibitem{K2016}{H. Kwon}, \emph{Similarity to the backward shift operator on the Dirichlet space}, J. Operator Theory, $\mathbf{76}$ (2016), no. 1, 133-140.

\bibitem{KT}{H. Kwon and S. Treil}, \emph{Similarity of operators and geometry of eigenvector bundles}, Publ. Mat. $\mathbf{53}$ (2009), no. 2, 417-438.

\bibitem{LM2023}{Y. C. Li and P. Ma}, \emph{On similarity and commutant of a class of multiplication operators on the Dirichlet space}, Banach J. Math. Anal. $\mathbf{17}$ (2023), no. 1, Paper No. 9.

\bibitem{QL}{Q. Lin}, \emph{Operator theoretical realization of some geometric notions}, Trans. Amer. Math. Soc. $\mathbf{305}$ (1988), no. 1, 353-367.


\bibitem{L2022}{S. Luo}, \emph{Corona theorem for the Dirichlet-type space}, J. Geom. Anal. $\mathbf{32}$ (2022), no. 3, 1-20.

\bibitem{M}{G. Misra}, \emph{Curvature inequalities and extremal properties of bundle shifts}, J. Operator Theory, $\mathbf{11}$ (1984), no. 2, 305-317.

\bibitem{VM1988}{V. M\"{u}ller}, \emph{Models for operators using weighted shifts}, J. Operator Theory, $\mathbf{20}$ (1988), no. 1, 3-20.

\bibitem{MV1993} {V. M\"{u}ller and F. H. Vasilescu}, \emph{Standard models for some commuting multioperators}, Proc. Am. Math. Soc. $\mathbf{117}$ (1993), no. 4, 979-989.

\bibitem{N1986}{N. K. Nikolski\v{\i}}, \emph{Treatise on the shift operator. Spectral function theory}, With an appendix by S. V. Hru\v{s}\v{c}ev and V. V. Peller, Translated from the Russian by Jaak Peetre. Grundlehren der mathematischen Wissenschaften $\mathbf{273}$, Springer-Verlag, Berlin, 1986.

\bibitem{MR1980}{M. Rosenblum}, \emph{A corona theorem for countably many functions}, Integral Equ. Oper. Theory, $\mathbf{3}$ (1980), no. 1, 125-137.

\bibitem{GCR1959}{G. C. Rota}, \emph{Note on the invariant subspaces of linear operators}, Rend. Circ. Mat. Palermo, $\mathbf{8}$ (1959), no. 2, 182-184.

\bibitem{GCR1960}{G. C. Rota}, \emph{On models for linear operators}, Comm. Pure Appl. Math. $\mathbf{13}$ (1960), 469-472.

\bibitem{SH}{A. L. Shields}, \emph{Weighted shift operators and analytic function theory}, in Math. Surveys, Amer. Math. Soc. $\mathbf{13}$ Providence, RI, 1974, 49-128.

\bibitem{DS1980} {D. Stegenga}, \emph{Multipliers of the Dirichlet space}, Illinois J. Math., $\mathbf{24}$ (1980), no. 1, 113-139.

\bibitem{NF1964}{B. Sz.-Nagy and C. Foia\c{s}}, \emph{Sur les contractions de l'espace de Hilert. VIII.}, Acta Sci. Math. $\mathbf{25}$ (1964), 38-71.

\bibitem{NF1970}{B. Sz.-Nagy and C. Foia\c{s}}, \emph{Harmonic analysis of operators on Hilbert space}, North-Holland Publi. Co., Amsterdam 1970.

\bibitem{VAT1981}{V. A. Tolokonnikov}, \emph{Estimates in the Carleson corona theorem, ideals of the algebra $H^{\infty}$, a problem of Sz.-Nagy}, Zap. Nauchn. Sem. Leningrad. Otdel. Mat. Inst. Steklov. (LOMI) $\mathbf{113}$ (1981), 178-198.

\bibitem{UM1980}{M. Uchiyama}, \emph{Corona theorems for countably many functions and estimates for their solutions}, preprint (1980).

\bibitem{UM1990}{M. Uchiyama}, \emph{Curvatures and similarity of operators with holomorphic eigenvectors}, Trans. Amer. Math. Soc. $\mathbf{319}$ (1990), no. 1, 405-415.

\bibitem{XJ1998}{J. Xiao}, \emph{The $\overline{\partial}$-problem for multipliers of the Sobolev space}, Manuscripta Math. $\mathbf{97}$ (1998), no. 2, 217-232.

\bibitem{ZKH}{K. Zhu}, \emph{Operators in Cowen-Douglas classes}, Ill. J. Math. $\mathbf{44}$ (2000), no. 4, 767-783.
\end{thebibliography}
\end{document}